\documentclass[11pt]{amsart}

\usepackage{amsfonts}
\usepackage{amsmath}
\usepackage{amssymb}
\usepackage{amscd}
\usepackage{mathrsfs}  % This package allows the use of script letter. Type \mathscr to invoke math script.
\usepackage{amsbsy}
\usepackage{amsthm}
\usepackage{graphicx}
\usepackage{fancyhdr}
\usepackage{epsfig}
\usepackage{color}
\usepackage{xy}
\usepackage{enumitem}
\usepackage[OT2,OT1]{fontenc}  % This package allows the use of cyrillic fonts. They are invoked by typing \textcyr{...}
\newcommand\cyr{%
\renewcommand\rmdefault{wncyr}%
\renewcommand\sfdefault{wncyss}%
\renewcommand\encodingdefault{OT2}%
\normalfont \selectfont} \DeclareTextFontCommand{\textcyr}{\cyr}

\newcommand{\be}{\begin{equation}}
\newcommand{\ee}{\end{equation}}

\newcommand{\bes}{\begin{equation*}}
\newcommand{\ees}{\end{equation*}}

\newcommand{\bH}{\mathbb{H}}

\newcommand{\N}{\mathbb{N}}

\newcommand{\R}{\mathbb{R}}

\newcommand{\X}{\mathbb{X}}

\newcommand{\cB}{\mathcal{B}}

\newcommand{\cF}{\mathcal{F}}
\newcommand{\mcC}{\mathcal{C}}

\newcommand{\cP}{\mathcal{P}}

\newcommand{\sI}{\mathscr{I}}

\newcommand{\sfE}{\mathsf{E}}
\newcommand{\sfF}{\mathsf{F}}

\newcommand{\sfV}{\mathsf{V}}
\newcommand{\sfX}{\mathsf{X}}
\newcommand{\sfY}{\mathsf{Y}}

\newcommand{\cC}{\mathcal{C}}

\renewcommand{\rmdefault}{cmr} % Arial
\renewcommand{\sfdefault}{cmr} % Arial
\newtheorem{theorem}{Theorem}[section]
\theoremstyle{plain}

\newtheorem{definition}{Definition}[section]
\newtheorem{example}{Example}[section]

\newtheorem{proposition}{Proposition}[section]
\newtheorem{remark}{Remark}[section]
\newtheorem{remarks}{Remarks}[section]

\newtheorem*{notation*}{Notation}
\numberwithin{equation}{section}

\DeclareMathOperator\Lip{Lip}

\DeclareMathOperator\esssup{{ess\, sup}}

\DeclareMathOperator\Sc{Sc}
\DeclareMathOperator\Ve{Vec}

\DeclareMathOperator\gr{graph}

\newcommand{\st}{ : }
\newcommand{\n}[1]{{\left\|{#1}\right\|}}
\newcommand{\abs}[1]{\left\vert{#1}\right\vert}

\begin{document}

\title[Fractal Interpolation]{Fractal Interpolation: From Global to Local, to Nonstationary and Quaternionic}

\author{Peter R. Massopust}
\address{Centre of Mathematics, Technical University of Munich, Boltzmannstr. 3, 85748 Garching b. Munich, Germany}
\email{massopust@ma.tum.de}

\begin{abstract}
We present an introduction to fractal interpolation beginning with a global set-up and then extending to a local, a non-stationary, and finally the novel quaternionic  setting. Emphasis is placed on the overall perspective with references given to the more specific questions.
\end{abstract}

\maketitle 

\section{Introduction}

Over the last decades, fractal interpolation and approximation have been extensively research. This research originated with \cite{B2} where a special set-up was used to define so-called affine fractal interpolation functions. The graphs of these affine fractal interpolation functions are the attractors of a class of iterated function systems and thus geometrically motivated. An analytic construction of general fractal functions originated in \cite{bedford,dubuc1,dubuc2} where the concept of a Read-Bajractarevi\'c operator is first encountered. Numerous constructions of fractal functions based on Read-Bajractarevi\'c operators satisfying given interpolation and approximation conditions followed. Some of these constructions are introduced and summarized in \cite{massopust,massopust1}. The number of publications in fractal interpolation theory is enormous and the interested reader may want to search for fractal functions using terms such as ``hidden variable,'' ``$V$-variable,'' ``coalescent,'' ``super,'' and ``$\alpha$-fractal functions'' to name just a few. 

The intend of this chapter is to introduce the reader to the concept of fractal interpolation and its extensions from a global to a local setting, then to non-stationarity and finally to a quaternionic setting. In a certain sense, these are the main set-ups with the possible exclusion of \emph{unbounded} fractal interpolation \cite{m7}. It is understandable that such an endeavor must necessarily restrict itself to the main points of each construction and setting. However, the exposition will give the reader an overall perspective of the issues involved and the techniques used, and can be used as a starting point for a deeper investigation into each the topics.

The outline of this chapter is as follows. Section 2 introduces the global setting of fractal interpolation and exhibits a relationship between the Read-Bajractarevi\'c operator and the solution to a canonically associated system of functional equations. In the next section, we extend global interpolation to a local set-up giving more flexibility to the construction. This type of interpolation found deep applications to fractal imaging and fractal compression \cite{BH}. In Section 4, the recently introduced concept of non-stationary fractal interpolation is presented and it is shown that backward trajectories allow distinct features to be delineated at different interpolation scales. The final Section 5, describes the novel setting of fractal interpolation in the theory of quaternions and shows that the non-commutative character of quaternions introduces even more intricate fractal patterns.

\section{Global Fractal Interpolation}\label{GFI}

The purpose of the current section is to introduce global fractal interpolation and to relate the global fractal interpolant to the solution of a system of functional equations. We see that this system of functional equations defines in a canonical way a Read-Bajractarevi\'c (RB) operator and vice versa. This relationship will be encountered several times in the subsequent sections as well.

In the following, $(\sfE,d_\sfE)$ denotes a normed space and $(\sfF,d_\sfF)$ a Banach space. For $n\in \N$, we write $\N_n :=\{1,\ldots, n\}$ for the initial segment of the natural numbers $\N$ of length $n$. 

For a given normed space $(\sfE, \n{\cdot}_\sfE)$ and a map $f: \sfE \to\sfE$, we define the Lipschitz constant associated with $f$ by
\[
\Lip (f) := \sup_{x,y \in \sfE, x \neq y} \frac{\n{f(x)-f(y)}_\sfE}{\n{x-y}_\sfE}.
\]
The map $f$ is called Lipschitz if $\Lip (f) < + \infty$ and a contraction (on $\sfE$) if $\Lip (f) < 1$.

\subsection{Bounded Solutions}
Let $\sfX$ be a nonempty bounded subset of $\sfE$. Suppose we are given a finite family $\{l_i\}_{i = 1}^{n}$ of injective contractions $\sfX\to \sfX$ generating a partition of $\sfX$ in the sense that
\begin{align}
&\forall\;i, j\in \N_n, i\neq j: l_i(\sfX)\cap l_j(\sfX) = \emptyset;\label{c1}\\
&\sfX = \bigcup_{i=1}^n l_i(\sfX).\label{c2}
\end{align}
For simplicity, we write $\sfX_i := l_i(\sfX)$.

Given the above set-up, we are looking for a global function $\psi:\sfX = \bigcup\limits_{i=1}^n \sfX_i\to\sfF$ satisfying $n$ functional equations of the form
\be\label{psieq}
\psi (l_i (x)) = q_i (x) + s_i (x) \psi (x), \quad\text{on $\sfX$ and for $i\in \N_n$},
\ee
where for each $i\in\N_n$, $s_i$ is a given bounded function $\sfX\to\R$ and $q_i$ a bounded function $\sfX\to\sfF$. Recall that a function $f:\sfX\to\sfF$ is called \emph{bounded} if there exists a finite $M>0$ such that $\n{f(x)}_\sfF \leq M$, for all $x\in \sfX$.

The idea is to consider \eqref{psieq} as the fixed point equation for an associated affine operator acting on an appropriately defined function space.

To this end, let $\cB (\sfX,\sfF) := \{f:\sfX\to\sfF : \text{$f$ is bounded}\}$ denote the the Banach space of bounded functions equipped with the supremums norm $\n{f} := \sup\limits_{x\in \sfX} \n{f(x)}_\sfF$. 

On the Banach space $\cB (\sfX,\sfF)$, we define an affine operator $T: \cB (\sfX,\sfF)\to \cB (\sfX,\sfF)$, called a Read-Bajractarevi\'c (RB) operator, by 
\be\label{eq3.17}
T f (x) = (q_i\circ l_i^{-1})(x) + (s_i\circ l_i^{-1})(x)\cdot (f\circ l_i^{-1})(x), 
\ee 
for $x\in \sfX_i$ and $i\in \N_n$, or, equivalently, by
\begin{align*}
T f (x) &= \sum_{i=1}^n (q_i\circ l_i^{-1})(x)\, \chi_{\sfX_i}(x) + \sum_{i=1}^n (s_i\circ l_i^{-1})(x)\cdot (f\circ l_i^{-1})(x)\, \chi_{\sfX_i}(x)\\
&= T(0) + \sum_{i=1}^n (s_i\circ l_i^{-1})(x)\cdot (f\circ l_i^{-1})(x)\, \chi_{\sfX_i}(x),\quad x\in \sfX,
\end{align*}
where $\chi_S$ denotes the characteristic function of a set $S$: $\chi_S(x) = 1$, if $x\in S$, and $\chi_S(x) = 0$, otherwise.

The following result is well-known (see, for instance, \cite{B2,massopust1}) but for the sake of completeness we reproduce the proof. We also refer the interested reader to \cite{SB} where a similar set-up is considered.

\begin{theorem}\label{sol}
The system of functional equations \eqref{psieq} has a unique bounded solution $\psi: \sfX\to\sfF$ provided that
\begin{enumerate}
\item conditions \eqref{c1} and \eqref{c2} are satisfied, and
\item $s:= \max\limits_{i\in \N_n} \sup\limits_{x\in \sfX} |s_i(x)| < 1$.
\end{enumerate} 
\end{theorem}

\begin{proof}
First note that, as the mappings $l_i$ are injective, the right-hand side of \eqref{psieq} can be written as the right-hand side of \eqref{eq3.17}.

As the functions $l_i$, $q_i$, and $s_i$ are all assumed to be bounded, $T$ maps $\cB (\sfX,\sfF)$ into itself. For all $f,g\in \cB (\sfX,\sfF)$, we have that
\begin{align*}
\sup_{x\in\sfX}\n{Tf(x) - Tg(x)}_\sfF &= \max_{i\in \N_n}\sup_{x\in\sfX_i} \n{(s_i\circ l_i^{-1})(x)\cdot (f-g)\circ l_i^{-1}(x)}_\sfF\\
&= \max_{i\in \N_n}\sup_{\xi\in\sfX} \n{s_i(\xi)\cdot (f-g)(\xi)}_\sfF\\
&\leq \max_{i\in \N_n}\sup\limits_{x\in \sfX} |s_i(x)| \sup_{x\in\sfX}\n{(f-g)(x)}_\sfF,
\end{align*}
from which it follows that
\[
\n{Tf-Tg} \leq s \n{f-g}.
\]
Hence, $T$ is contractive on the Banach space $\cB (\sfX,\sfF)$ and therefore, by the Banach Fixed Point Theorem, has a unique fixed point $\psi\in\cB (\sfX,\sfF)$. This fixed point solves the functional equations \eqref{psieq}.
\end{proof}

\begin{remarks}\hfill
\begin{enumerate}
\item The fixed point $\psi\in\cB (\sfX,\sfF)$ of the RB operator $T$ is also called a \emph{bounded fractal function}. In this context, Eqn. \eqref{psieq} is also referred to as a \emph{self-referential equation} for $\psi$. 
\item The self-referential equation $T\psi = \psi$ expresses the fractal nature of the $\gr\psi$: It is made up of a finite number of copies of itself with each copy being supported on the partitioning sets $\sfX_i$. Hence, the terminology \emph{fractal function} for $\psi$.
\item The proof of Banach's Fixed Point Theorem also provides an algorithm for the construction of $\psi$: Choose \emph{any} function $\psi_0\in \cB (\sfX,\sfF)$ and iteratively define the following sequence of functions: 
$$
\psi_k := T \psi_{k-1}, \quad k\in \N. 
$$
Then, $\psi$ is given by $\psi = \lim\limits_{k\to\infty}\psi_k$ where the limit is taking with respect to the norm $\n{\cdot}$ on  $\cB (\sfX,\sfF)$.
\item The afore-mentioned algorithm for the construction of $\psi$ together with the proof of the Banach Fixed Point Theorem gives an error estimate as well, namely,
\[
\n{\psi - \psi_{k}} \leq \frac{s^k}{1-s} \n{\psi_1 - \psi_0}, \quad k\in \N.
\]
\item The fixed point $\psi$ depends on $n$, the partition $(\sfX_i : i\in\N_n)$, and the functions $s_i$ and $q_i$ with different choices yielding different fractal functions.
\item Emphasizing the dependence of $\psi$ on the functions $s_i$, the expression $s$-fractal function can be found in the literature. (See, for instance, \cite{N}.) In this context, one considers a fractal function as the image under an operator $\cF^s$ associating with a given (non-fractal) function its fractal analogue.
\item Functional equations such as \eqref{psieq} exhibit connections to so-called \emph{fractels} \cite{bhm2,massopust1} and also to the approximation of rough functions \cite{BHVV}.
\end{enumerate}
\end{remarks}

Conditions \eqref{c1} and \eqref{c2} cannot be relaxed without adding some compatibility conditions to guarantee that the RB operator $T$ has the form given by Eqn. \eqref{eq3.17}. Should Eqn. \eqref{c1} not be satisfied, one would have to impose in our current setting the following compatibility conditions:
\begin{align}\label{wd}
\forall x_1,&x_2\in\sfX:\nonumber\\
& \; l_i (x_1) = l_j (x_2)\;\;\Longrightarrow\;\; q_i(x_1) + s_i (x_1) \psi (x_1) = q_j(x_2) + s_ j(x_2) \psi (x_2).
\end{align}
We refer to \cite{SB,SB1} for more details regarding this issue.\\

As an application of the above approach to obtain solutions to functional equations of the form \eqref{psieq} or equivalently finding the unique fixed points of the associated RB operator \eqref{eq3.17}, we provide the following example.
\begin{example}\label{ex1}
Let $\sfE := \R =: \sfF$ together with the Euclidean norm $\abs{\cdot}$. Further. let $\sfX := [0,1)\subset\sfE$. Assume that we are given two injective contractions $l_i:[0,1)\to [0,1)$, $i=1,2$, with $l_1(x) := \frac13 x$ and $l_2(x) := \frac23 x + \frac13$. Hence, $\sfX_1 = [0,\frac13)$ and $\sfX_2 = [\frac13,1)$. Clearly, $\sfX = \sfX_1\cup\sfX_2$ and $\sfX_1\cap\sfX_2 = \emptyset$. 

Now choose $q_1(x) = -1$, $q_2(x) = {x}$, $s_1(x) = \frac12 \sin(x)$, and $s_2(x):=-\frac23 \cos(x)$. The system of functional equations and the associated RB operator read then
\[
\psi(\tfrac13 x) = -1 + \tfrac12 \sin(x)\psi(x)\quad\text{and}\quad\psi(\tfrac23 x + \tfrac13) = {x} - \tfrac23\cos(x)\psi(x),
\]
and
\[
Tf(x) = \begin{cases} -1 + \frac12\sin(3x) f(3x), & 0\leq x < \tfrac13;\\
{3x-1} - \frac23\cos(\tfrac12(3x-1))f(\tfrac12(3x-1)), & \tfrac13 \leq x < 1,
\end{cases}
\]
respectively. 

As $s = \frac23 < 1$, $T$ is contractive. A plot of the solution, respectively, fixed point $\psi$, is shown in Fig. 1.
\begin{figure}[h!]
\begin{center}
\includegraphics[width=7cm, height= 4cm]{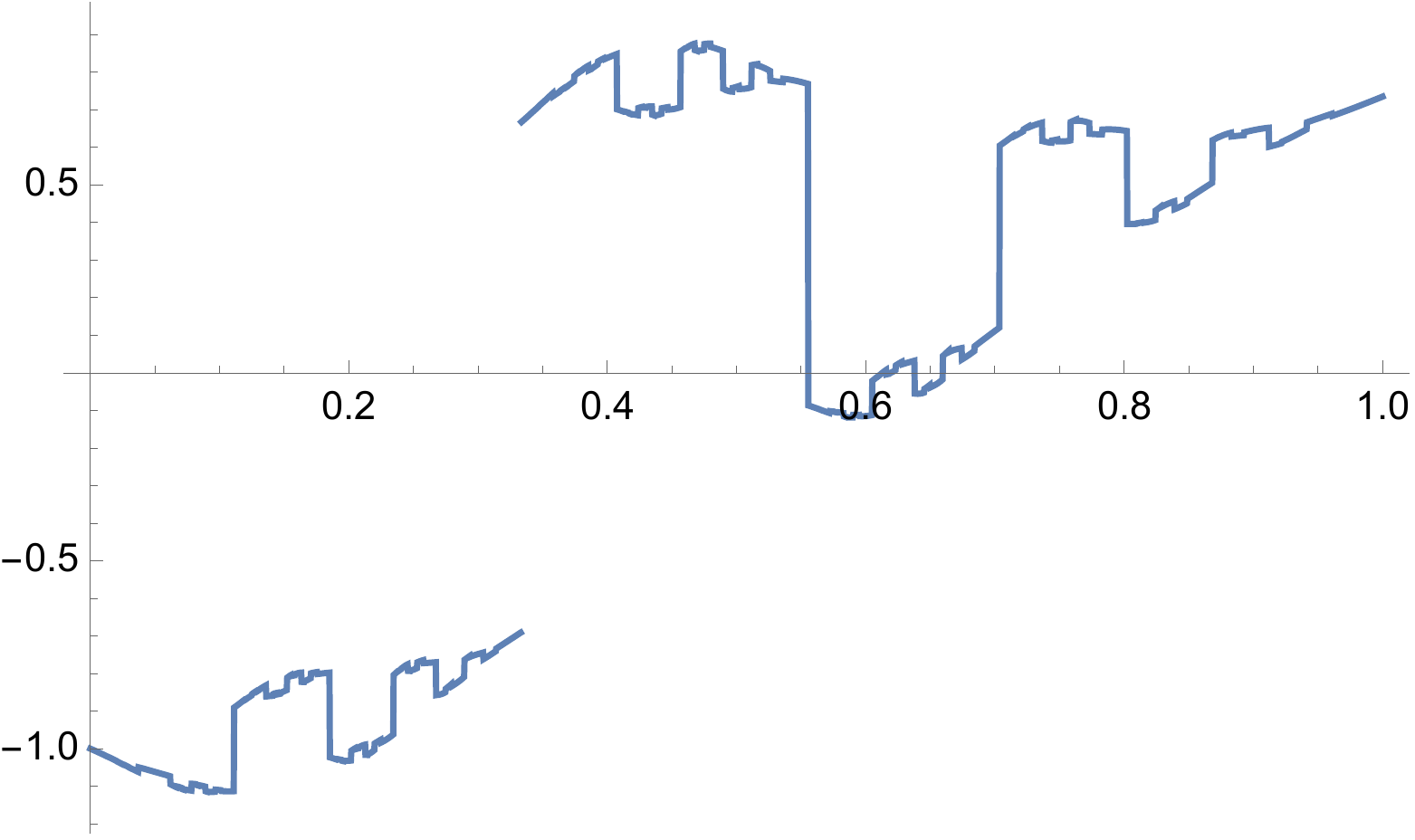}
\caption{The solution/fixed point $\psi$.}\label{fig1}
\end{center}
\end{figure}
\end{example}
\subsection{$L^p$ solutions}
In the following, we set $\sfX\subset\sfE:=\R^m$ and $\sfY:=\R^k$ where the Euclidean spaces $\R^m$ and $\R^k$ are endowed with the corresponding canonical Euclidean norms.

Recall that the (real) Lebesgue spaces $L^p (\sfX, \R^k)$, where $\sfX\subset \R^m$ is nonempty, are defined as consisting of (equivalence classes of) functions $f:\sfX\to\R^k$ for which
\[
\n{f}_p :=\begin{cases}\left(\displaystyle{\int_\sfX \n{f(x)}^p dx} \right)^{1/p}, & 1\leq p < \infty;\\ \\
\esssup_{x\in\sfX} \n{f(x)}, & p = \infty.
\end{cases}
\]
is finite. Here, $\n{f(x)} := \sqrt{\abs{f_1(x)}^2 + \cdots + \abs{f_k(x)}^2}$ with $f := (f_1, \ldots, f_k)$.

We ask under what conditions on the functions $q_i$ and $s_i$ the solution $\psi$ is an element of $L^p(\sfX, \R^k)$, for $1\leq p < \infty$ and a bounded nonempty set $\sfX\subset\R^m$. 

To this end, note that in order for $\psi$ to be in $L^p(\sfX, \R^k)$, the RB operator $T$ must map $L^p(\sfX, \R^k)$ into itself. Therefore, the functions $q_i$ and $s_i$ must also be in $L^p(\sfX, \R^k)$. Moreover, $s_i$ needs to be in $L^\infty(\sfX, \R^k)$ for the product $s_i\cdot f$ to be in $L^p(\sfX, \R^k)$. Thus, as $\sfX$ is bounded it has finite measure and therefore $s_i\in L^\infty(\sfX, \R^k)$ implies that $s_i\in L^p(\sfX, \R^k)$ for all $1\leq p \leq \infty$.

Now it remains to be shown that the RB operator $T$ is contractive on $L^p(\sfX, \R^k)$. For this purpose, let $f,g\in L^p(\sfX, \R^k)$. Then, with $\sfX_i:= l_i(\sfX)$,
\begin{align*}
\n{Tf - Tg}_p^p &= \int_\sfX \n{Tf (x) - Tg(x)}^p dx\\
&= \int_\sfX \n{\sum_{i=1}^n s_i(l_i^{-1}(x))\cdot (f-g)(l_i^{-1}(x)) \chi_{\sfX_i}(x)}^p dx\\
&\leq \sum_{i=1}^n \int_{\sfX_i} \abs{s_i(l_i^{-1}(x))}^p \n{(f-g)(l_i^{-1}(x))}^p dx\\
&= \sum_{i=1}^n \int_{\sfX} \abs{(l_i^{-1})'(x)} \abs{s_i(x)}^p \n{(f-g)((x)}^p dx.
\end{align*}

If, for all $i\in \N_n$, $\n{l_i^{-1})'}_\infty =: \lambda_i<\infty$ and $\n{s_i}_\infty = :s_i<\infty$, then 
\[
\n{Tf - Tg}_p^p \leq \left(\sum_{i=1}^n \lambda_i s_i^p\right) \n{(f-g)((x)}_p^p,
\]
and $T$ is contractive provided that
\[
\sum_{i=1}^n \lambda_i s_i^p < 1.
\]
Hence, we arrived at the following result
\begin{theorem}\label{th3.2}
The system of functional equations
\[
\psi (l_i (x)) = q_i (x) + s_i (x) \psi (x), \quad\text{on $\sfX\subset\R^m$ and for $i\in \N_n$},
\]
has a unique solution $\psi\in L^p(\sfX, \R^k)$, $1\leq p \leq \infty$, respectively, the RB operator
\[
T f (x) = (q_i\circ l_i^{-1})(x) + (s_i\circ l_i^{-1})(x)\cdot (f\circ l_i^{-1})(x),\quad{x\in\sfX_i},\;i\in \N_n,
\]
a unique fixed point $\psi\in L^p(\sfX, \R^k)$ provided that
\begin{enumerate}
\item $q_i\in L^p(\sfX, \R^k)$, $s_i\in L^\infty (\sfX, \R^k)$ and 
\item $\sum\limits_{i=1}^n \lambda_i s_i^p < 1$, where $\lambda_i = \n{(l_i^{-1})'}_\infty $ and $s_i = \n{s_i}_\infty$.
\end{enumerate}
\end{theorem}

\begin{remark}
In a similar fashion, one can derive conditions such that the unique solutions/fixed points $\psi$ are elements of H\"older or Sobolev spaces. See, for instance, \cite{PRM,massopust,massopust1}.
\end{remark}
\subsection{Continuous Solutions}
So far, we only considered bounded solution/fixed points $\psi$ for \eqref{psieq} and \eqref{eq3.17}. However, in some instances, a continuous or even differentiable solution is required. We only present a result for continuous $\psi$ and make some remarks about how to obtain differential solutions. 

\begin{theorem}\label{th3.3}
The system of functional equations \eqref{psieq} has a unique continuous solution $\psi: \sfX\to\sfF$ provided that
\begin{enumerate}
\item $\sfX = \bigcup\limits_{i=1}^n l_i(\sfX)$,
\item the functions $l_i$, $q_i$, and $s_i$ are continuous,
\item and for all $i,j\in \N_n$ and $x_1, x_2\in X$:
\be\label{joinupc}
\lim_{x\to x_1} f_j(x) = f_i(x_2)\;\;\Longrightarrow\;\; \lim_{x\to x_1} q_j(x) + s_j(x) \psi(x) = q_i(x_2) + s_i(x_2) \psi(x_2).
\ee
\end{enumerate} 
\end{theorem}
\begin{proof}
We refer the interested reader to \cite{SB} or \cite{massopust}. In the former reference, the proof follows the functional equation setting and in the latter the RB operator setting.
\end{proof}

\begin{example}

\end{example}
We connect up with the previous Example 1 but choose as $\sfX := [0,1]$. We modify the functions $q_i$ to be $q_1(x) := x$ and $q_2(x) := 1-x$. but keep $s_1$ and $s_2$ unchanged. Note that here we have $\sfX_1 \cap \sfX_2 = \{\frac13\}$ and we need to ensure that conditions \eqref{wd} are satisfied. In particular, as we have $l_1(1) = \frac13 = l_2(0)$, the following equality has to hold:
\be\label{3.7}
q_1(1) + s_1(1)\psi(1) = q_2(0) + s_2(0)\psi(0).
\ee
The functional equations \eqref{psieq} imply for $x\in \{0,1\}$
\[
\psi(0) = q_1 (0) + s_1(0) \psi(0)\quad\text{and}\quad \psi(1) = q_2 (1) + s_2(1)\psi(0),
\]
which gives the values of $\psi$ at the endpoints of $\sfX$:
\[
\psi(0) = \frac{q_1(0)}{1-s_1(0)}\quad\text{and}\quad \psi(1) = \frac{q_2(1)}{1-s_2(1)}.
\]
The validity of \eqref{3.7} guarantees the existence of a \emph{bounded} solution $\psi$ (since $s = \frac23 < 1$). As $q_1(0) = 0 = q_2(1)$, the solution $\psi$ also vanishes on the boundary of $\sfX$: $\psi(0) = 0 = \psi(1)$.

In order to obtain a \emph{continuous} solution, equations \eqref {joinup} must be satisfied. In our current setting, as all the functions involved are continuous on $\sfX$, we obtain
\[
\lim_{x\to 0-} q_2(x) + s_2(x)\psi(x) = q_1(1) + s_1(1)\psi(1),
\]
which is identical to \eqref{3.7}. 

The solution/fixed point $\psi$ is therefore continuous. The graph of $\psi$ is depicted in Figure \ref{fig2}.
\begin{figure}[h!]
\begin{center}
\includegraphics[width=7cm, height= 4cm]{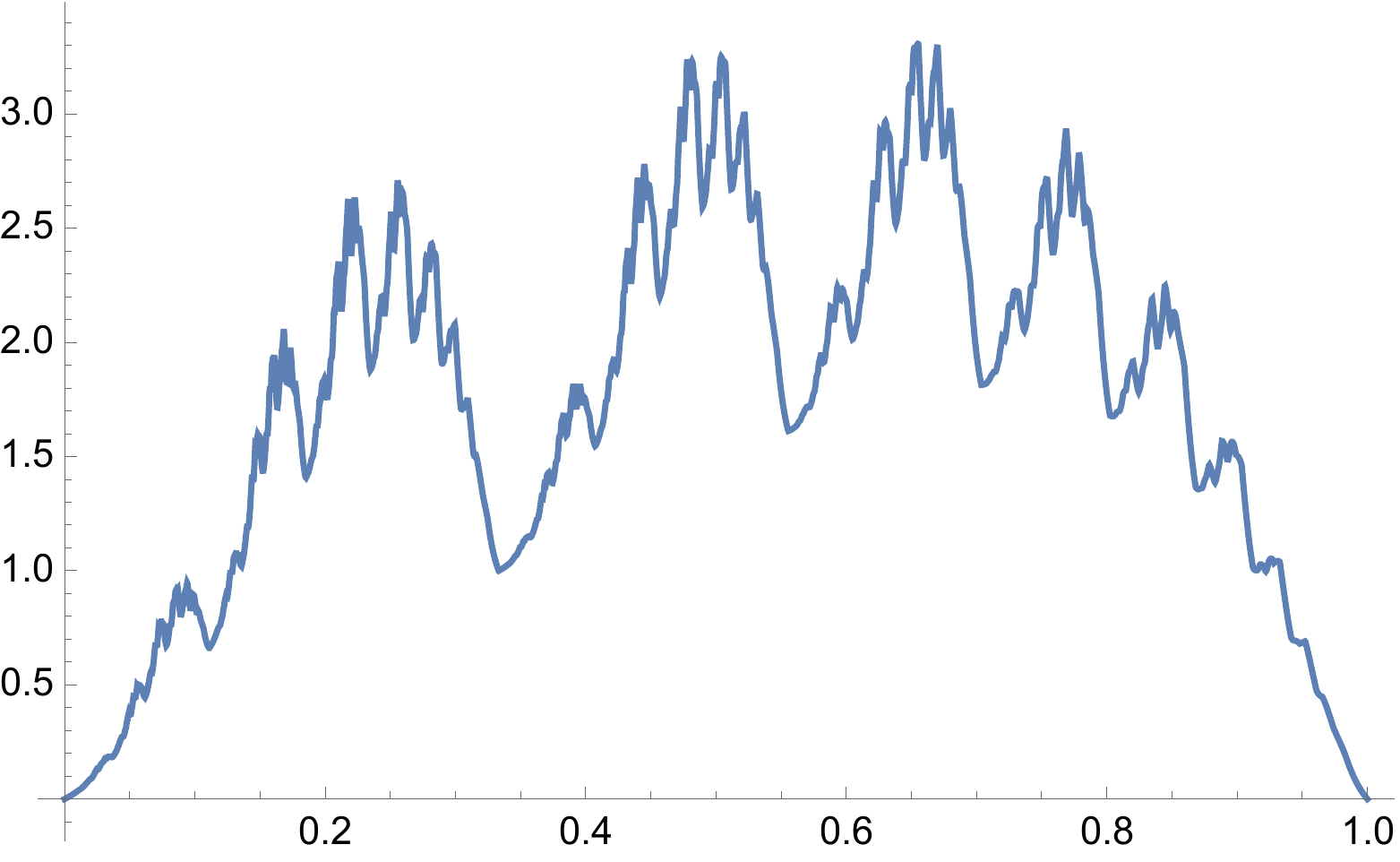}
\caption{A continuous solution/fixed point $\psi$.}\label{fig2}
\end{center}
\end{figure}

In \cite{SB}, a similar setting is considered.
\begin{remark}
For many applications, in particular the setting where $\sfE := \R$, the partition $(\sfX_i : i\in \N_n)$ is induced by a given finite set of data points $\{(x_i, y_i)\in \R\times\sfF : x_0 < x_1 < \cdots < x_n\}$ with $\sfX_i := [x_{i-1},x_i]$. The injective contractions $l_i$ are then affine mappings $[x_0,x_n]\to [x_{i-1},x_i]$. For $\sfF := \R$, this type of continuous fractal interpolation was first introduced in \cite{B2} and the fixed point $\psi$ was termed a \emph{fractal interpolation function} as it was also required that $y_j = \psi(x_j)$, $j\in\{0,1,\ldots, n\}$.
\end{remark}
\section{Local Fractal Interpolation}
In this section, we introduce a generalization of global fractal interpolation. Previously, we considered a fixed nonempty bounded subset $\sfX$ of a normed space $\sfE$ generating a partition of $\sfX$ via a finite family of injective contractions. Now, we replace the single subset $\sfX$ by a finite family of subsets $\sfX_i$ of $\sfX$.

More precisely, let $\{\sfX_i : i \in \N_n\}$ be a family of nonempty subsets of a fixed subset $\sfX$ of a normed space $\sfE$. Suppose $\{l_i\}_{i = 1}^{n}$ is a collections of injective mappings from $\sfX_i\to \sfX$ generating a partition of $\sfX$ in the sense that
\begin{align}
&\forall\;i, j\in \N_n, i\neq j: l_i(\sfX_i)\cap l_j(\sfX_j) = \emptyset;\label{cl1}\\
&\sfX = \bigcup_{i=1}^n l_i(\sfX_i).\label{cl2}
\end{align}
Note that the $l_i$ need not be contractive mappings here. 
\begin{remark}
One can actually have $m<n$ subsets of $\sfX$ and still be able to define $n$ injective mappings satisfying \eqref{cl1} and \eqref{cl2}. In this case some of the injections $l_i$ share the same domain $X_j$, i.e., $\sfX_j$ is repeated a certain number of times $\sfX_j = \sfX_{j+1} = \cdots = \sfX_{j+r}$. This situation occurs in Section \ref{sect4.3}.
\end{remark}
Local fractal interpolation looks for local solutions $\psi: \sfX = \bigcup\limits_{i\in \N_n} l_i (\sfX_i)\to \sfF$ of functional equations or for fixed points of RB operators of the form
\be\label{localpsi}
\psi (l_i (x)) = q_i (x) + s_i (x) \psi (x), \quad\text{$x\in\sfX_i$,\;\;$i\in \N_n$},
\ee
respectively,
\be\label{Tlocal}
T f (x) = (q_i\circ l_i^{-1})(x) + (s_i\circ l_i^{-1})(x)\cdot (f_i\circ l_i^{-1})(x), \quad\text{$x\in l_i(\sfX_i)$,\;\; $i\in \N_n$},
\ee 
where $f_i := f\vert_{\sfX_i}$, on appropriate function spaces.
\subsection{Bounded Local Solutions.}
The extension of the results presented in Section \ref{GFI} carry over to the setting of local fractal interpolation. However, special care must be taken when considering the domains of the functions involved. 

Besides the function space $\cB (\sfX,\sfF)$ already introduced in the previous section, we also need the local version of this space, namely, $\cB (\sfX_i,\sfF)$, $i\in \N_n$. To this end and in view of the result below, we now assume that
\begin{enumerate}
\item $s_i\in \cB(\sfX_i,\R)$ and
\item $q_i\in \cB(\sfX_i,\sfF)$.
\end{enumerate}
Then the RB operator $T$ defined in \eqref{Tlocal} maps $\cB (\sfX,\sfF)$ into itself. Hence, we arrive at the local version of Theorem \ref{sol}.
\begin{theorem}
The system of functional equations \eqref{localpsi} has a unique bounded solution $\psi: \sfX\to\sfF$, respectively, the RB operator \eqref{Tlocal} has a unique bounded fixed point $\psi: \sfX\to\sfF$ provided that
\begin{enumerate}
\item conditions \eqref{cl1} and \eqref{cl2} are satisfied, and
\item $s:= \max\limits_{i\in \N_n} \sup\limits_{x\in \sfX_i} |s_i(x)| < 1$.
\end{enumerate} 
\end{theorem}

\begin{proof}
The proof follows the same arguments as those given in the proof of Theorem \ref{sol}; replace $\sfX_i$ and $\sfX$ there by $l_i(\sfX_i)$ and $\sfX_i$ here. We leave the details to the reader who may also consult \cite{SB}.
\end{proof}

As in the global setting of fractal interpolation, a few remarks are in order.
\begin{remarks}\hfill
\begin{enumerate}
\item The solution/fixed point $\psi$ is referred to as a \emph{bounded local fractal function}.
\item The function $\psi$ depends on $n$, the family of subsets $\{\sfX_i : i\in\N_n\}$, the partition $(l_i(\sfX) : i\in\N_n)$ induced by the injective mappings $l_i$, and the now locally defined functions $s_i$ and $q_i$.
\end{enumerate}
\end{remarks}

As observed above, conditions \eqref{cl1} and \eqref{cl2} cannot be relaxed without adding some compatibility conditions to guarantee that the RB operator $T$ has the form given by Eqn.~ \eqref{Tlocal}. Should Eqn. ~\eqref{cl1} not be satisfied, one would have to impose in our current setting the following, now local, compatibility conditions:
\begin{align}
\forall i,j\in &\N_n\,\forall x_1\in\sfX_i\,\forall x_2\in\sfX_j:\nonumber\\
& \; l_i (x_1) = l_j (x_2)\;\;\Longrightarrow\;\; q_i(x_1) + s_i (x_1) \psi (x_1) = q_j(x_2) + s_ j(x_2) \psi (x_2).
\end{align}
We again refer to \cite{SB,SB1} for more details regarding this issue.
\subsection{$L^p$ Local Solutions}
As an application of local fractal functions for a particular setting, we again consider local fractal functions in $L^p$ spaces for $p\in[0,\infty]$.

For this purpose, we choose the following set up. Let $\sfE := \R := \sfY$ with the canonical Euclidean norm and let $\sfX:= [0,1]$. Suppose that we are given a partition of $\sfX$ of the form $\Delta := (0 =: x_0 < x_1 < \cdots < x_{n-1} < x_n := 1)$, for some integer $n>1$. Furthermore, suppose that $\{\sfX_i \st i \in \N_n\}$ is a family of half-open intervals of $[0,1]$. 

We define affine mappings $l_i:\sfX_i$ onto $[x_{i-1}, x_i)$, $i = 1, \ldots, n-1$, and from $\sfX_n^+ := \sfX_n\cup l_n^{-1}(1-)$ onto $[x_{n-1},x_n]$, where $l_n$ maps $\sfX_n$ onto $[x_{n-1}, x_n)$.

We have the following result for RB-operators defined on the Lebesgue spaces $L^p[0,1]$, $1\leq p \leq \infty$. (See, also, \cite{bhm}.)

\begin{theorem}\label{th4.2}
Assume that $q_i\in L^p (\sfX_i, [0,1])$ and $s_i\in L^\infty (\sfX_i, \R)$, $i \in\N_n$. 
The system of functional equations
\[
\psi (l_i (x)) = q_i (x) + s_i (x) \psi (x), \quad x\in [0,1],\; i\in \N_n,
\]
has a unique solution $\psi\in L^p[0,1]$, $1\leq p \leq \infty$, respectively, the RB operator
\be\label{Top}
T f (x) = (q_i\circ l_i^{-1})(x) + (s_i\circ l_i^{-1})(x)\cdot (f_i\circ l_i^{-1})(x),\quad{x\in\sfX_i},\;i\in \N_n,
\ee
%a unique fixed point $\psi\in L^p(\Omega, \R^k)$ provided that
%
%The operator $\Phi: L^p [0,1]\to \R^{[a,b]}$ defined by
%\be\label{Phi}
%\Phi g := \sum_{i=1}^N (\lambda_i \circ u_i^{-1})\,\chi_{u_i(\X_i)} + \sum_{i=1}^N (S_i\circ u_i^{-1})\cdot (g_i\circ u_i^{-1})\,\chi_{u_i(\X_i)},
%\ee
has a unique fixed point $\psi\in L^p(\Omega, \R^k)$, where $f_i = f\vert_{\X_i}$, provided that
\be\label{condition}
\begin{cases}
\left(\displaystyle{\sum_{i=1}^n}\, a_i \,\|s_i\|_{\infty, \sfX_i}^p\right)^{1/p} < 1, & p\in[1,\infty);\\ 
\max\limits_{i\in\N_n}\|s_i\|_{\infty,\sfX_i} < 1, & p = \infty,
\end{cases}
\ee
where $a_i$ denotes the Lipschitz constant of $(l_i^{-1})'$. Here, we wrote $\|s_i\|_{\infty,\sfX_i}$ for $\sup\limits_{x\in \sfX_i} |s_i(x)|$.
\end{theorem}

\begin{proof}
Note that under the hypotheses on the functions $q_i$ and $s_i$ as well as the mappings $l_i$, $T f$ is well-defined and an element of $L^p[0,1]$. It remains to be shown that under condition \eqref{condition}, $T$ is contractive on $L^p[0,1]
$. 

To this end, let $g,h \in L^p [0,1]$ and let $p\in [0,\infty)$. Then
\begin{align*}
\|T g - T h\|^p_{p} & = \int\limits_{[0,1]} |T g (x) - T h (x)|^p dx\\
& = \int\limits_{[0,1]} \left|\sum_{i=1}^{n} (s_i\circ l_i^{-1})(x) [(g_i\circ l_i^{-1})(x) - (h_i\circ l_i^{-1})(x)]\,\chi_{l_i(\X_i)}(x)\right|^p\, dx\\
& = \sum_{i=1}^{n}\,\int\limits_{[x_{i-1},x_i]}\left| (s_i\circ l_i^{-1})(x) [(g_i\circ l_i^{-1})(x) - (h_i\circ l_i^{-1})(x)]\right|^p\,dx\\
& = \sum_{i=1}^{n}\,a_i\,\int\limits_{\X_i} \left| s_i (x) [g_i(x)- h_i(x)]\right|^p\,dx\\
&  \leq \sum_{i=1}^{n}\,a_i\,\|s_i\|^p_{\infty, \X_i}\,\int\limits_{\X_i} \left| g_i(x) - h_i(x)\right|^p\,dx = \sum_{i=1}^{n}\,a_i\,\|s_i\|^p_{\infty, \X_i}\,\|f_i - g_i\|^p_{p,\X_i}\\
& = \sum_{i=1}^{n}\,a_i\,\|s_i\|^p_{\infty, \X_i}\,\|g_i - h_i\|^p_{p} \leq \left(\sum_{i=1}^{n}\,a_i\,\|s_i\|^p_{\infty, \X_i}\right) \|g - h\|^p_{p}.
\end{align*}
Now let $p= \infty$. Then,
\begin{align*}
\|T g - T h\|_{\infty} & = \left\|\sum_{i=1}^{n} (s_i\circ l_i^{-1})(x) [(g_i\circ l_i^{-1})(x) - (h_i\circ l_i^{-1})(x)]\,\chi_{l_i(\X_i)}(x)\right\|_\infty\\
& \leq \max_{i\in\N_n}\,\left\| (s_i\circ l_i^{-1})(x) [(g_i\circ l_i^{-1})(x) - (h_i\circ l_i^{-1})(x)]\right\|_{\infty,\X_i}\\
& \leq \max_{i\in\N_n}\|s_i\|_{\infty,\X_i} \left\|g_i - h_i]\right\|_{\infty,\X_i} = \max_{i\in\N_n}\|s_i\|_{\infty,\X_i} \left\|g_i - h_i]\right\|_{\infty}\\
&  \leq \left(\max_{{i\in\N_n}}\,\|s_i\|_{\infty,\X_i}\right) \left\|g - h]\right\|_{\infty}
\end{align*}
These calculations prove the claims.
\end{proof}
\begin{remark}
Although conditions \eqref{condition} resembles those presented in Theorem \ref{th3.2}, they are -- because of the local nature of the estimates -- more subtle than the former ones.
\end{remark}
\subsection{Continuous Local Solutions}\label{sect4.3}
As in the global case, we can consider continuous local solutions to the given set of functional equations, respectively, an RB operator. Here, we present the analog to Theorem \ref{th3.3} but refer to the literature for the proof. A good reference in the functional equation setting is again \cite{SB}. In the RB operator setting, we refer to \cite{m2}.

\begin{theorem}\label{th4.3}
The system of functional equations \eqref{localpsi} has a unique continuous solution $\psi: \sfX\to\sfF$ provided that
\begin{enumerate}
\item $\sfX = \bigcup\limits_{i=1}^n l_i(\sfX_i)$,
\item the functions $l_i$, $q_i$, and $s_i$ are continuous,
\item and $\forall i,j\in \N_n$, $i\neq j$, $\forall x_1\in\sfX$, $\forall x_2\in \sfX_i$:
\be\label{joinup}
\lim_{\substack{x\to x_1\\ x\in\sfX_j}} f_j(x) = f_i(x_2)\;\;\Longrightarrow\;\; \lim_{\substack{x\to x_1\\ x\in\sfX_j}} q_j(x) + s_j(x) \psi(x) = q_i(x_2) + s_i(x_2) \psi(x_2).
\ee
\end{enumerate} 
\end{theorem}
As an application and an example for a continuous local fractal interpolation, we present the following set up which plays an important role in fractal-based numerical analysis as discussed in \cite{bhm}. 

Suppose that $\sfE := \R =: \sfF$ and $\sfX:= [0,1]$. For an even integer $n\in\N$, define subsets
\be\label{Xs}
\sfX_{2j-1} := \sfX_{2j} := \left[\tfrac{2j-2}{n},\tfrac{2j}{n}\right], \quad j \in \{1,\ldots, \tfrac{n}{2}\}
\ee
and affine mappings $l_i: \sfX_i\to [0,1]$ by 
\be\label{ls}
l_{2j-1} (x) := \tfrac{x}{2} + \tfrac{j-1}{n}\quad\text{and}\quad l_{2j} (x) := \tfrac{x}{2} + \tfrac{j}{n}, \quad x\in \sfX_{2j-1} = \sfX_{2j}.
\ee
Note that
\[
l_i (\sfX_i) = \left[\tfrac{i-1}{n},\tfrac{i}{n}\right], \quad i\in\N_n.
\]
Further, let $x_i := l_i(\sfX_i)\cap l_{i+1}(\sfX_{i+1}) = \{\frac{i}{n}\}$, $i\in\N_{n-1}$, $x_0:= 0$, and $x_n := 1$. In the terminology of \cite{SB}, the elements of $\{x_i : i\in \N_{n-1}\}$ are called \emph{contact points}.

We denote the distinct endpoints of the partitioning intervals $\{l_i(\X_{i})\}$ by $\{x_0 < x_1 <\ldots < x_{N}\}$ where $x_0 = 0$ and $x_N = 1$ and refer to them as {knots}.

Furthermore, we assume that we are given interpolation values at the endpoints of the intervals $\X_{2j-1} = \X_{2j}$:
\be\label{Delta}
\Delta := \left\{(x_{2j}, y_j)\st j = 0, 1, \ldots, n/2\right\}.
\ee
Let 
$$
\cC_{\Delta}(\sfX) := \{f\in \cC(\sfX) : f(x_{2j}) = y_j, \,\forall\, j = 0, 1, \ldots, n/2\}.
$$
Here, $\cC(\sfX) := \cC(\sfX,\R)$ denotes the Banach space of all continuous functions $\sfX\to\R$ endowed with the supremum norm $\n{\cdot}$. Note that $\cC_{\Delta}(\sfX)$ is a closed, hence complete, \emph{metric} subspace of $\cC(\sfX)$ to which we can apply the Banach fixed point theorem. To this end, consider an RB operator $T$ of the form \eqref{Top} acting on $\cC_{\Delta}(\sfX)$. 

In order for $T$ to map $\cC_{\Delta}(\sfX)$ into itself we require that 
\be\label{qsandss}
q_i, s_i \in \cC(\sfX_i) := \cC(\sfX_i,\R) := \{f: \sfX_i\to \R\st \text{$f$ continuous}\}
\ee
and that
\be\label{intcon}
y_{j-1} = (T f) (x_{2(j-1)}) \quad \wedge \quad y_{j} = (T f) (x_{2j}), \quad j = 1, \ldots, n/2,
\ee
where $x_{2j} := \frac{2j}{n}$. 

We remark that the preimages of the knots $x_{2(j-1)}$ and $x_{2j}$ are the endpoints of $\sfX_{2j-1} = \sfX_{2j}$. Substitution of $T$, as given in \eqref{Top}, into \eqref{intcon} and simplification results in
\be\label{intcons}
\begin{split}
q_{2j-1} (x_{2(j-1)}) + \left(s_{2j-1}(x_{2(j-1)}) - 1\right) y_{j-1} & = 0,\\
q_{2j} (x_{2j}) + \left(s_{2j}(x_{2j}) - 1\right) y_{j} & = 0.
\end{split}
\ee

To ensure global continuity of $T f$ on $\sfX=[0,1]$, we also have to impose the following join-up conditions at the oddly indexed knots. (Note that these oddly indexed knots are the images of the midpoints of the intervals $\sfX_{2j-1} = \sfX_{2j}$.)
\be\label{prejoin}
(T f) (x_{2j-1}-)  = (T f) (x_{2j-1}+) , \quad j = 1, \ldots, n/2.
\ee
These join-up conditions imply that
\be\label{4.16}
q_{2j} (x_{2(j-1)}) + s_{2j} (x_{2(j-1)}) y_{j-1} = q_{2j-1} (x_{2j}) + s_{2j-1} (x_{2j})  y_j.
\ee
In the case that all functions $q_i$ and $s_i$ are constant, \eqref{4.16} reduces to the condition given in \cite[Example 2]{bhm}. 

We summarize these results in the next theorem.
\begin{theorem}\label{thm8}
Let $\sfX:= [0,1]$ and let $n\in 2\N$. Suppose that subsets of $\sfX$ are given by \eqref{Xs} and the associated mappings $l_i$ by \eqref{ls}. Further suppose that the functions $q_i$ and $s_i$ satisfy \eqref{qsandss} and that the join-up conditions \eqref{intcon} and \eqref{prejoin} hold. Then, an RB operator $T$ of the form \eqref{Top} maps $\cC_\Delta (\sfX)$ into itself and is well-defined. 

If, in addition, 
\be\label{condfors}
\max\left\{\|s_i\|_{\infty,\sfX_i} : i\in\N_n\right\} < 1,
\ee
then $T$ is a contraction on $\cC_\Delta (\sfX)$ and thus possesses a unique continuous fixed point $\psi:[0,1]\to \R$ satisfying $\psi\in \cC_\Delta (\sfX)$.
\end{theorem}

This unique fixed point $\psi$ is called a \emph{continuous local fractal function}.

\begin{proof}
It remains to show that under the condition \eqref{condfors}, $T$ is contractive on $\cC_\Delta (\sfX)$. This, however, follows immediately from the case $p=\infty$ in the proof of Theorem \ref{th4.2}.
\end{proof}

For RB operators of the form \eqref{Top} mapping function spaces such as H\"older, Sobolev, Besov or Triebel Lizorkin into themselves, we refer the interested reader to \cite{m2,m3}.
\section{Non-stationary Fractal Interpolation}\label{sect5}
In this section, we extend the notion of global fractal interpolation to a non-stationary setting. In other words, we no longer assume that we keep the functions $q_i$ and $s_i$ the same at each level of iteration in the construction of the fixed point $\psi$.
\subsection{The non-stationary setting}
To this end, consider a doubly-indexed family of injective contractions $\{l_{i_k,k} : i_k\in \N_{n_k}, \, k\in \N\}$ from $\sfX\to \sfX$ where $\sfX$ is nonempty bounded subset of a normed space $\sfE$ generating a partition of $\sfX$ for each $k\in \N$ in the sense of \eqref{cl1} and \eqref{cl2}. 

Suppose that $\sfF$ is a Banach space, $\{q_{i_k,k}: i_k\in \N_{n_k}, \, k\in \N\} \subset \cB(\sfX,\sfF)$, and $\{s_{i_k,k}: i_k\in \N_{n_k}, \, k\in \N\}\subset \cB(\sfX,\R)$ are such that 
\[
s := \sup\limits_{k\in \N} \max\limits_{i_k\in \N_k} \|s_{i_k,k}\|_\infty < 1.
\] 
For each $k\in \N$, we define an RB operator $T_k: \cB(\sfX,\sfF)\to \cB(\sfX,\sfF)$ by
\begin{align}
(T_k f)(l_{i_k,k}(x)) &:= q_{i_k,k}(x) + s_{i_k,k}(x) \cdot f(x), \quad\forall x\in\sfX.\label{nonstat1}
\end{align}
It is not difficult to verify that each $T_k$ is a contraction on $\cB(\sfX,\sfF)$ with Lipschitz constant 
\be\label{Lip}
\Lip (T_k) = \max\limits_{i_k\in \N_k} \|s_{i_k,k}\|_\infty \leq s < 1.
\ee
In order to continue, we require the following definition and result from \cite{LDV} adapted to our current setting.

\begin{definition}{\cite[Definition 3.6]{LDV}}
Let $\{T_k\}_{k\in \N}$ be a sequence of transformations $T_k:\cB(\sfX,\sfF)\to \cB(\sfX,\sfF)$. A subset $\sI$ of $\cB(\sfX,\sfF)$ is called an invariant set of the sequence $\{T_k\}_{k\in \N}$ if
\[
\forall\,k\in \N\;\forall\,x\in \sI: T_k (x)\in \sI.
\]
\end{definition}
A criterion for obtaining an invariant domain for a sequence $\{T_k\}_{k\in \N}$ of transformations on $\cB(\sfX,\sfF)$ is also given in \cite{LDV}.

\begin{proposition}{ \cite[Lemma 3.7]{LDV}}\label{prop2.1}
Let $\{T_k\}_{k\in \N}$ be a sequence of transformations on $\cB(\sfX,\sfF)$. Suppose there exists a $g\in \cB(\sfX,\sfF)$ such that for all $f\in\cB(\sfX,\sfF)$
\[
\n{T_k(x) - g} \leq \mu\,\n{f - g} + M,
\]
for some $\mu\in [0,1)$ and $M> 0$. Then the ball $B_r(g)$ of radius $r = M/(1-\mu)$ centered at $g$ is an invariant set for $\{T_k\}_{k\in \N}$.
\end{proposition}
\begin{proof}
The proof, although in a more general setting, is found in \cite{LDV}.
\end{proof}

\begin{proposition}\label{prop5.2}
Let $\{T_k\}_{k\in \N}$ be a sequence of RB operators of the form \eqref{nonstat1} on $(\cB(\sfX,\sfF),\n{\cdot})$. Suppose that the elements of $\{q_{i_k,k}: i_k\in \N_{n_k}, \, k\in \N\}$ satisfy
\be\label{condq}
\sup\limits_{k\in \N} \max\limits_{i_k\in \N_k} \n{q_{i_k,k}} \leq M,
\ee
for some $M > 0$. Then the ball $B_r(0)$ of radius $r=M/(1-s)$ centered at $0\in \cB(\sfX,\sfF)$ is an invariant set for $\{T_k\}_{k\in \N}$.
\end{proposition}
\begin{proof}
%Note that since $Y$ is an $F$-space, we have for all $a,b\in Y$, 
%\[
%d_Y(a+b,0) \leq d_Y(a+b,b) + d_Y(b,0) = d_Y(a,0) + d_Y(b,0).
%\]
Let $x\in \sfX$. Then there exists an $i_k\in \N_{n_k}$ with $x\in l_{i_k,k}(\sfX)$. Thus, for any $f\in \cB(\sfX,\sfF)$, 
\begin{align*}
\n{T_k f (x)}_\sfF &\leq \n{s_{i_k,k}\circ l_{i_k,k}^{-1}(x) \cdot f\circ l_{i_k,k}^{-1}(x)}_\sfF + \n{T_k(0)}_\sfF 
\end{align*}
By \eqref{condq},  $T_k(0)$ is uniformly bounded on $\cB(\sfX,\sfF)$ by $M > 0$. Hence, 
\[
\n{s_{i_k,k}\circ l_{i_k,k}^{-1}(x) \cdot f\circ l_{i_k,k}^{-1}(x)}_\sfF \leq s \n{f\circ l_{i_k,k}^{-1}(x)}_\sfF,
\]
which gives, after taking the supremum over $x\in\sfX$, $\n{T_k f} \leq s\,\n{f} + M$. Proposition \eqref{prop2.1} now yields the statement. 
\end{proof}

To arrive at the main result of this section, we require yet another definition. (See also \cite[Section 4]{LDV}.)
\begin{definition}
Suppose that $f_0\in\cB(\sfX,\sfF)$ and that $\{T_k\}_{k\in \N}$ be a sequence of RB operators on $\cB(\sfX,\sfF)$. The sequences
\be
\Phi_k (f_0) := T_k\circ T_{k-1} \circ \cdots \circ T_1 (f_0)
\ee
and
\be
\Psi_k (f_0) := T_1\circ T_2 \circ \cdots \circ T_k (f_0)
\ee
are called the forward, respectively, backward trajectory of $f_0$.
\end{definition}

In connection to the above definition, we need the following result from \cite[Corollary 11]{LDV} adapted to our setting.

\begin{theorem}\label{th5.1}
Suppose that $\{T_k\}_{k\in \N}$ is a sequence of RB operators of the form \eqref{nonstat1} on $\cB(\sfX,\sfF)$. Further suppose that
\begin{enumerate}
\item there exists a nonempty closed invariant set $\sI\subseteq\cB(\sfX,\sfF)$ for $\{T_k\}_{k\in \N}$;
\item and
\be\label{eq5.6}
\sum_{k=1}^\infty\prod_{j=1}^k \Lip (T_j) < \infty.
\ee
\end{enumerate}
Then the backward trajectories $\Psi_k(f_0)$ converge for any initial $f_0\in\sI$ to a unique function $\psi\in\sI$.
\end{theorem}

With these preliminaries, we obtain the main result of this section.

\begin{theorem}\label{thm4.1}
The backwards trajectories $\{\Psi_k\}_{k\in \N}$ converge for any initial $f_0\in \sI$ to a unique function $\psi \in \sI$, where $\sI$ is the closed ball in $\cB(\sfX,\sfF)$ of radius $M/(1-s)$ centered at $0$.
\end{theorem}
\begin{proof}
By Proposition \ref{prop5.2} and Theorem \ref{th5.1} it remains to be shown that \eqref{eq5.6} is satisfied. This, however, follows directly from \eqref{Lip}:
\[
\prod_{j=1}^k \Lip (T_j) \leq s^k\quad\text{and}\quad\sum_{k=1}^\infty s^k = \frac{s}{1-s} < \infty.\qedhere
\]
\end{proof}

A fixed point $\psi$ generated by a sequence $\{T_k\}$ of non-stationary RB operators will be called a \emph{non-stationary fractal function (of class $\cB(X,Y)$).}

Employing a non-stationary sequence of RB operators $\{T_k\}$ allows the construction of more general fractal functions exhibiting different local behavior at different scales. This is illustrated by the following example which is taken from \cite{m4}.

\begin{example}
Let $\sfX:= [0,1]\subset\R$ and $\sfF:= \R$. Consider the two RB operators
\[
T_1 f (x) := \begin{cases} 2x + \frac12 f(2x), & x\in [0,\frac12),\\
2 - 2x + \frac12 f(2x-1), & x\in [\frac12,1],
\end{cases}
\]
and 
\[
T_2 f (x) := \begin{cases} 2x + \frac14 f(2x), & x\in [0,\frac12),\\
2 - 2x + \frac14 f(2x-1), & x\in [\frac12,1].
\end{cases}
\]
For both operators, $l_i (x) := \frac12 (x + i-1)$, $i=1,2$.

The sequence $T_1^k f \to \tau$, where $\tau$ denotes the Takagi function \cite{Tak} and $T_2^k\to q$, where $q(x)= 4x(1-x)$.

Consider the alternating sequence $\{T_k\}_{k\in \N}$ of RB operators given by
\[
T_k := \begin{cases} T_1, & 10(j-1) < k \leq 10j-5,\\
T_2, & 10j-5 < k \leq 10j,
\end{cases}\quad j\in \N.
\]
Two images of this hybrid attractor of the backward trajectory $\Psi_k$ starting with $f_0 \equiv 0$ are shown in Figure \ref{fig11}.
\begin{figure}[h!]
\begin{center}
\includegraphics[width=5cm, height=3cm]{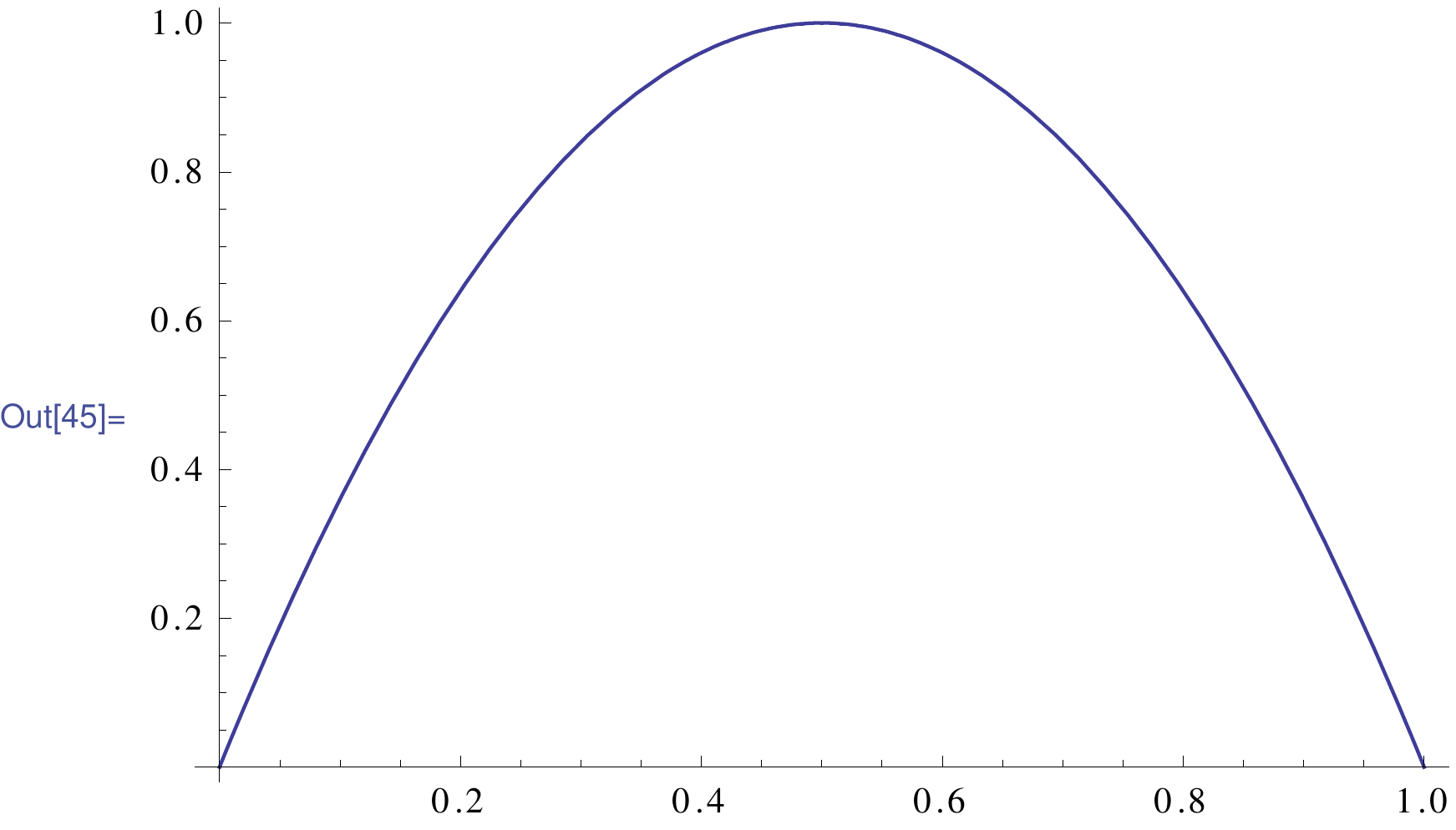}\hspace*{1cm}\includegraphics[width=5cm, height=3cm]{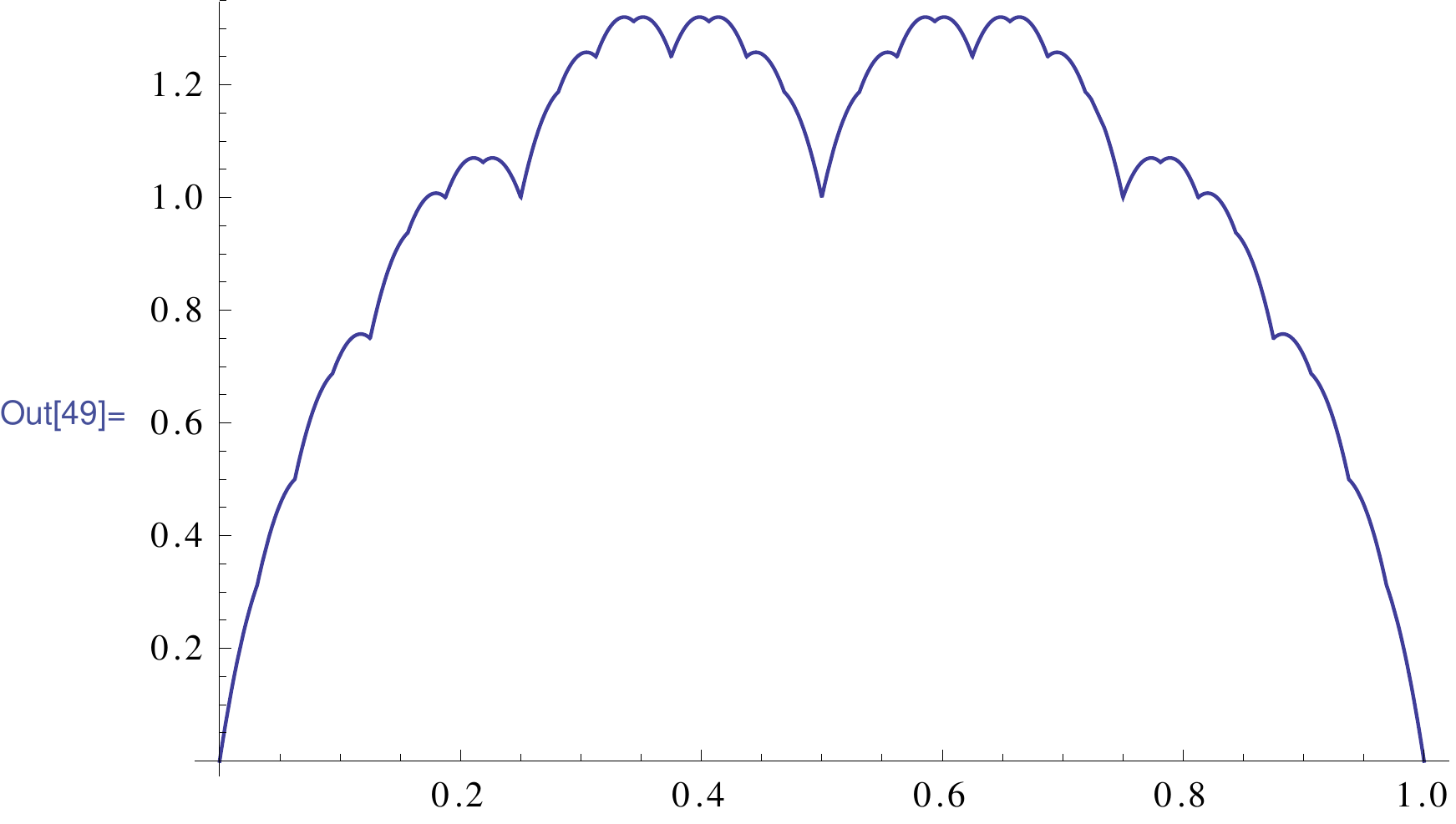}
\end{center}
\caption{The hybrid $\tau-q$ attractor. It is smooth at one scale but fractal at another.}\label{fig11}
\end{figure}
\end{example}

\subsection{Non-stationary Fractal Interpolation}

In this subsection, we consider the case $\sfX := [0,1]$ and $\sfF := \R$. In the following, we use a sequence $\{T_k\}$ of RB operators of a particular form to obtain a continuous fixed point $\psi$. 

To this end, consider an RB operator defined by \eqref{eq3.17}, choose two arbitrary functions $f,b\in \cB(\sfX) := \cB(\sfX, \R)$ and set
\be\label{qs}
q_i := f\circ l_i - s_i\cdot b.
\ee
The (stationary) RB operator associated with this particular choice is then given by
\be
(Tg)(x) = f + s_i(l_i^{-1}(x))\cdot(g-b)(l_i (x)), \quad x\in l_i(\sfX)
\ee
\begin{remark}
The fixed point of the above defined RB operator may be thought of as the ``fractalization" of a given function $f$.
\end{remark}

To work in a non-stationary setting, let $k\in \N$ and let $\{l_{i_k,k} : i_k\in \N_{n_k}, \, k\in \N\}$ be family of injections from $[0,1]\to [0,1]$ generating a partition of $[0,1]$ in the sense of \eqref{cl1} and \eqref{cl2}. We assume w.l.o.g. that $l_{1,k}(0) = 0$ and $l_{n_k,k}(1) = 1$ and define
\begin{align*}
x_{i_k-1,k} := l_{i_k,k}(0), \quad x_{i_k,k} := l_{i_k,k}(1), \quad i_k\in \N_{n_k}
\end{align*}
where $x_{0,k} :=0$ and $x_{n_k,k} := 1$. By relabelling -- if necessary -- we may further assume that $0 = x_{0,k}  < \cdots < x_{i_k-1,k} < x_{i_k,k} < \cdots x_{n_k,k} = 1$.

Let $f\in \cC(\sfX)$ be arbitrary. Define a metric subspace of $\cC(\sfX)$ by 
\[
\cC_*(\sfX) := \{g\in \cC(\sfX) : g(0) = f(0) \wedge g(1) = f(1)\}
\] 
and note that $\cC_*(\sfX)$ becomes a complete linear metric space when endowed with the metric $d$ that is induced by the sup-norm from $\cC(\sfX)$. 

Furthermore, let $b\in \mcC_*[0,1]$ be the unique affine function whose graph connects the points $(0, f(0))$ and $(1, f(1))$:
\be\label{b}
b(x) = (f(1) - f(0)) x + f(0).
\ee
Let $\{\cP_k\}_{k\in \N}$ be a family of sets of points in $\sfX\times \sfF$ where 
\[
\cP_k :=\{(x_{j_k}, f(x_{j,k})\in \sfX\times \sfF : j = 0,1,\ldots, n\}.
\]
For $k\in \N$, define an RB operator $T_k: \mcC_*[0,1]\to \mcC_*[0,1]$ by
\be\label{eq5.1}
T_k g = f + \sum_{i_k=1}^{n_k} s_{i_k,k}\circ l_{i_k,k}^{-1} \cdot (g-b)\circ l_{i_k,k}^{-1}\, \chi_{l_{i_k,k} [0,1]},
\ee
where $\{s_{i_k,k}\}_{i_k=1}^{n_k}\subset \cC[0,1]$ such that 
\[
\sup_{k\in \N} \max_{i_k\in \N_{i_k}} \|s_{i_k,k}\|_\infty < 1.
\]
Note that 
\begin{align*}
T_k g (x_{i_k,k}-) &= f(x_{i_k,k}-) + s_{i_k,k}\circ l_{i_k,k}^{-1}(x_{i_k,k}-) \cdot (g-b)\circ l_{i_k,k}^{-1}(x_{i_k,k}-) \\
&= f(x_{i_k,k}) + s_{i_k,k} (1)\cdot (f - b)(1) = f(x_{i_k,k})
\end{align*}
and
\begin{align*}
T_k g (x_{i_k,k}+) &= f(x_{i_k,k}+) + s_{i_k+1,k}\circ l_{i_k+1,k}^{-1}(x_{i_k,k}+) \cdot (g-b)\circ l_{i_k+1,k}^{-1}(x_{i_k,k}+) \\
&= f(x_{i_k,k}) + s_{i_k+1,k} (0)\cdot (f - b)(0) = f(x_{i_k,k}).
\end{align*}
implying that $T_k g$ is continuous at the points $x_{i_k,k}\in [0,1]$:
\[
T_k g (x_{i_k,k}-) = T_k g (x_{i_k,k}+), \quad\forall\,i_k\in\{1, \ldots, n-1\}.
\]
Hence,  $T_k g\in \cC_*[0,1]$ and $T_k g$ interpolates $\cP_k$ in the sense that
\[
T_k g (x_{i_k,k}) = f(x_{i_k,k}), \quad \forall\,i_k\in \N_{n_k}.
\]

\begin{proposition}
A nonempty closed invariant set for $\{T_k\}_{k\in \N}$ is given by the closed ball in $\cC_*(\sfX)$,
\be\label{I}
\sI = \left\{g\in \mcC_*[0,1] : \|g\| \leq \frac{\|f\| + s \|b\|}{1-s}\right\},
\ee
where $s$ is defined by \eqref{Lip}.
\end{proposition}

\begin{proof}
From the form \eqref{qs} of the functions $q_{i_k,k}$, we obtain from \eqref{condq} the estimate $\|q_{i_k,k}\| \leq \|f\| + s \|b\|$, which by Proposition \ref{prop2.1} yields the result.
\end{proof}

Together with Theorem \ref{thm4.1}, the above arguments prove the next theorem.

\begin{theorem}
Let $\{T_k\}_{k\in \N}$ be a sequence of RB operators of the form \eqref{eq5.1} each of whose elements acts on the complete metric space $(\cC_*(\sfX), d)$ where $f\in \cC_*(\sfX)$ is arbitrary and $b$ is given by \eqref{b}. Furthermore, let the family of functions $\{s_{i_k,k}\} \subset \cC(\sfX)$ satisfy \eqref{Lip}. Then, for any $f_0\in \sI$, the backward trajectories $\Psi_k (f_0)$ converge to a function $\psi \in \sI$ which interpolates $\cP_k$.
\end{theorem}

We refer to the fixed point $\psi\in \cC_*(\sfX)$ as a \emph{continuous non-stationary fractal interpolation function}.

\begin{remark}
As $f_0$ one may choose $f$ or $b$.
\end{remark}

To illustrate the above results, we present the following example from \cite{m4}.

\begin{example}
Consider the two RB operators $T_i: C[0,1]\to C[0,1]$, $i=1,2$, defined by
\[
(T_1 f)(x) = \begin{cases} -\frac{1}{2}\,f(4x), & x\in[0,\frac{1}{4}),\\
 -\frac{1}{2} + \frac{1}{2}\,f(4x-1), & x\in[\frac{1}{4},\frac{1}{2}),\\
 \frac{1}{2}\,f(4x-2), & x\in[\frac{1}{2},\frac{3}{4}),\\
 \frac{1}{2} + \frac{1}{2}\,f(4x-3), & x\in[\frac{3}{4},1],\end{cases}
\]
and 
\[
(T_2 f)(x) := \begin{cases} \frac34 f(2x), & x \in [0,\frac12),\\
\frac34 + \frac14 f(2x-1), &  x \in [\frac12,1].
\end{cases}
\]
The RB operators $T_1$ and $T_2$ generate \emph{Kiesswetter's fractal function} \cite{Kiess}, respectively, a \emph{Casino function} \cite{DS}. 

Consider again the alternating sequence $\{T_i\}_{i\in \N}$ of RB operators given by
\[
T_k := \begin{cases} T_1, & 10(j-1) < k \leq 10j-5,\\
T_2, & 10j-5 < k \leq 10j,
\end{cases}\quad j\in \N.
\]
Two images of the hybrid attractor of the backward trajectory $\Psi_k$ starting with the function $f_0 (x) = x$, $x\in\sfX$, are shown below in Figure \ref{fig22}.
\begin{figure}[h!]
\begin{center}
\includegraphics[width=5cm, height=4cm]{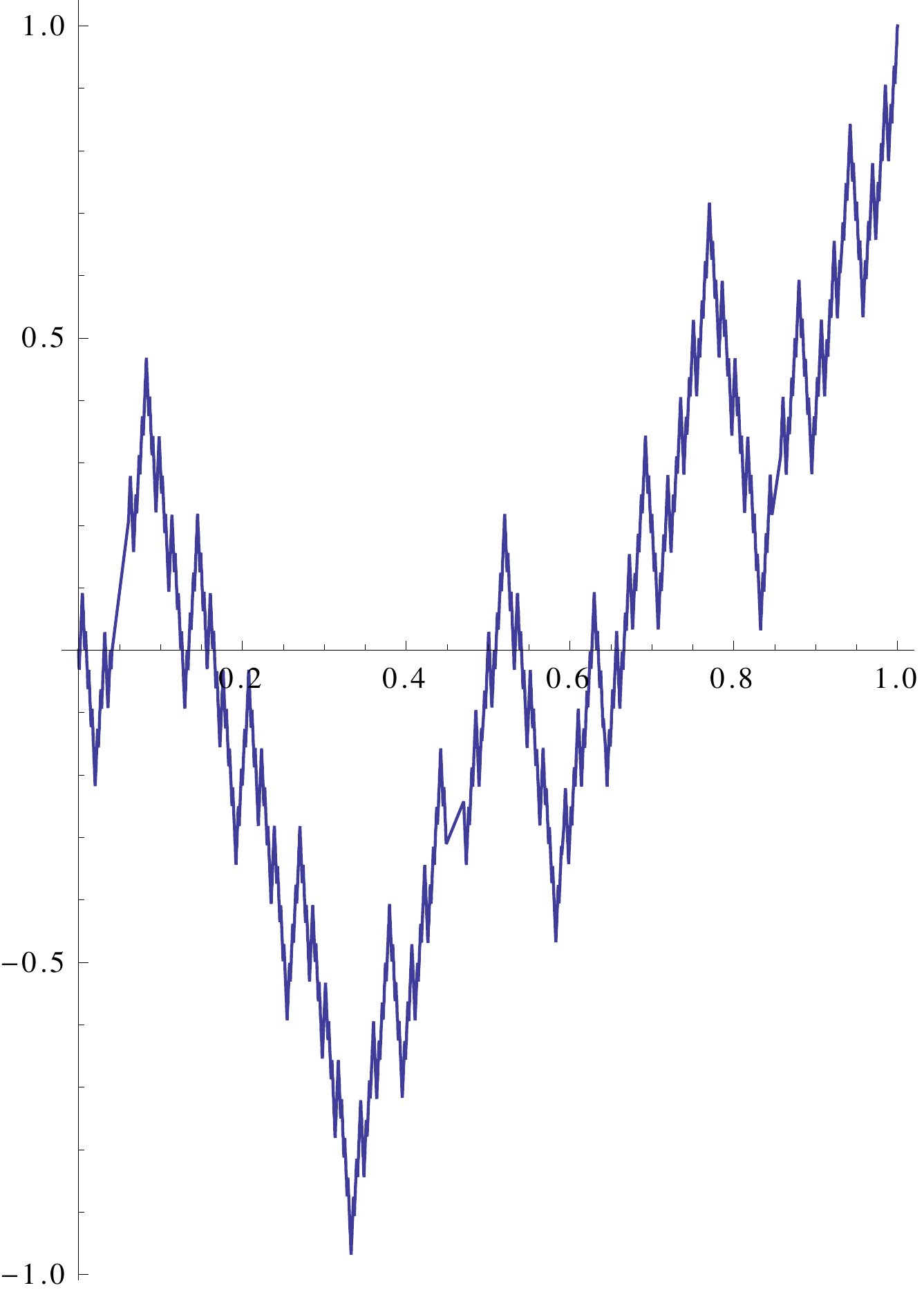}\hspace*{1cm}\includegraphics[width=5cm, height=4cm]{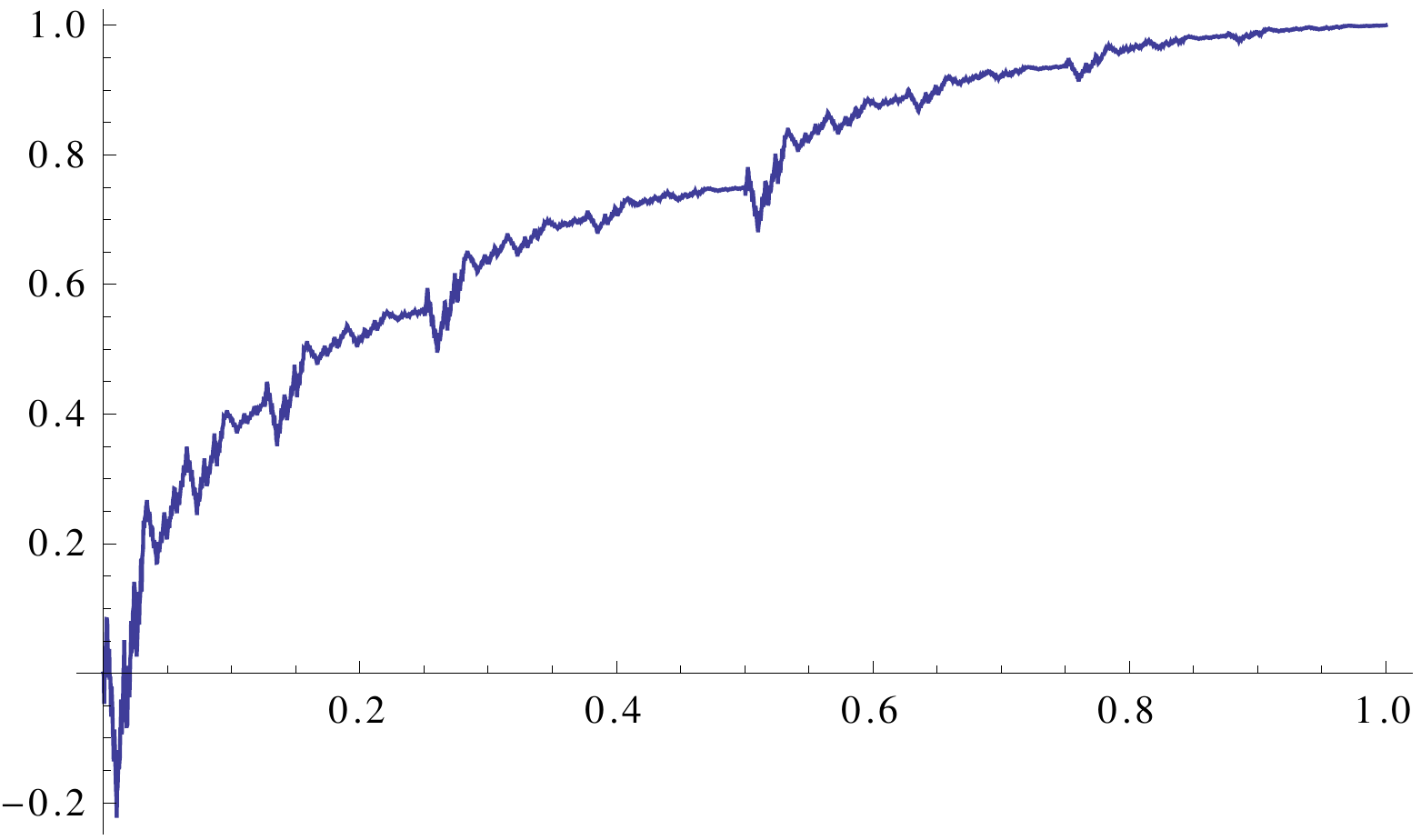}
\end{center}
\caption{The hybrid Kiesswetter-Casino attractor.}\label{fig22}
\end{figure}

\end{example}

\section{Quaternionic Fractal Interpolation}
In this section, we extend fractal interpolation to a quaternionic setting. As quaternions from a non-commutative division algebra, the non-commutativity generates more intricate fractal patterns. 

First, we give a short introduction to quaternions and present some of those properties that are relevant for the remainder of this section. The interested reader is referred to the literature on the subject, a short and subjective list of which is \cite{Bour,GHS,GS,Krav,MGS}.

\subsection{A Brief Introduction to Quaternions}
Let $\{e_1, e_2, e_3\}$ be the canonical basis of the Euclidean vector space $\R^3$. We call $\{e_1, e_2, e_3\}$ imaginary units and require that the following multiplication rules hold:
\begin{gather}
e_1^2=e_2^2=e_3^2=-1,\label{6.1}\\ 
e_1e_2=e_3=-e_2e_1,\quad e_2e_3=e_1=-e_3e_2,\quad e_3e_1=e_2=-e_1e_3.\label{6.2}
\end{gather}
\begin{remark}
Note that \eqref{6.2} is equivalent to $e_1e_2e_3 = -1$.
\end{remark}
A \emph{real quaternion $q$} is then an expression of the form
\[
q = a+\sum_{i=1}^3v_ie_i, \quad a, v_1, v_2, v_3\in\R.
\]
The addition and multiplication of two quaternions $q_1 = a+\sum_{i=1}^3v_ie_i$ and $q_2 = b+\sum_{i=1}^3w_ie_i$ is defined by
\begin{align*}
q_1 + q_2 &:= (a+b) + \sum_{i=1}^3 (v_i+w_i) e_i\\
q_1 q_2&:= (ab - v_1w_1 - v_2w_2 - v_3w_3)\\
&\;\;+(a {w_1}+b {v_1}+{v_2} {w_3}-{v_3} {w_2})e_1\\
&\;\;+(a {w_2}+b {v_2}-{v_1} {w_3}+{v_3} {w_1})e_2\\
&\;\;+(a {w_3}+b {v_3}+{v-1}{w_2}-{v_2} {w_1})e_3.
\end{align*}
Each quaternion $q=a+\sum\limits_{i=1}^3v_ie_i$ may be decomposed as $q=\Sc (q)+\Ve (q)$ where $\Sc (q)=a$ is the {\it scalar part} of $q$ and $\Ve (q)=v=\sum\limits_{i=1}^3v_ie_i$ is the {\it vector part} of $q$. $v=\Ve (q)$ is also called a \emph{quaternionic vector}.

The {\it conjugate} $\overline{q}$ of the real quaternion $q=a+v$ is the quaternion $\overline{q}=a-v$. Note that $q\overline{q}=\overline{q}q=a^2+|v|^2=a^2+\sum\limits_{i=1}^3v_i^2$. Therefore, we can define a norm on $\bH$ by setting
\[
\abs{q} := \sqrt{q\overline{q}}.
\]
The inverse of a quaternion $q$ is given by
\[
q^{-1} = \frac{\overline{q}}{\abs{q}^2}.
\]

It is straight-forward to establish that the collection of all real quaternions
\[
\bH := \bH_\R := \left\{a+\sum_{i=1}^3v_ie_i : a, v_1, v_2, v_3\in{\mathbb R}\right\},
\]
is a four-dimensional associative normed division algebra over $\R$. Due to the multiplication rules \eqref{6.1} and \eqref{6.2}, $\bH$ is not commutative. We also note that $\bH$, as already indicated above, is a four-dimensional vector space over $\R$ with basis $\{e_0,e_1,e_2,e_3\}$, where $e_0 := 1$.

\begin{remark}
We note that $\bH$ and $\R^4$ are identical as \emph{point sets} but differ in their algebraic structure. The vector space $\R^4$ is not an algebra whereas $\bH$ is one, albeit non-commutative.
\end{remark}

\begin{remark}
As we can define a norm and therefore a metric on $\bH$, $\bH$ becomes a topological space where open sets are defined via the metric. All other topological concepts such as limits, convergence, compactness etc. then follow. $\bH$ thus becomes a topological space and, moreover, also a complete metric space.
\end{remark}

We also remark that if $v=\sum\limits_{j=1}^3v_je_j$ and $w=\sum\limits_{j=1}^3w_je_j$ are quaternionic vectors, then 
\begin{equation}
vw=-\langle v,w\rangle +v\wedge w\label{vec mult},
\end{equation}
where $\langle v,w\rangle :=\sum\limits_{j=1}^3v_jw_j$ is the scalar product of $v$ and $w$ and 
$$v\wedge w :=(v_2w_3-v_3w_2)e_1+(v_3w_1-v_1w_3)e_2+(v_1w_2-v_2w_1)e_3$$ 
is the vector (cross) product of $v$ and $w$.

For our purposes, we need to introduce the analog of a Banach space in the quaternionic setting. To this end, we begin with the following definitions. (See, also \cite{Bour}.)
\begin{definition}
A real vector space $\sfV$ is called a left quaternionic vector space if it is a left $\bH$-module, i.e., if there exists a mapping $\bH\times\sfV\to \sfV$, $(q,v)\mapsto q v$ which satisfies
\begin{enumerate}
\item $\forall v\in \sfV\,\forall q_1,q_2\in \bH:$ $(q_1+q_2)v = q_1 v + q_2 v$.
\item $\forall v_1,v_2\in \sfV\,\forall q\in \bH:$ $q (v_1+v_2) = q v_1 + q v_2$.
\item $\forall v\in \sfV\,\forall q_1,q_2\in \bH:$ $q_1(q_2 v) = (q_1 q_2) v$.
\end{enumerate}
\end{definition}

\begin{remark}
In analogous fashion, one defines a \emph{right} quaternionic vector space as a right $\bH$-module where the mapping is now $\sfV\times\bH\to V$, $(v,q)\mapsto v q$.
\end{remark}

A \emph{two-sided quaternionic vector space} $\sfV$ is a left and right quaternionic vector space such that $\lambda v = v \lambda$, for all $\lambda\in\R$ and for all $v\in\sfV$. An example of a two-sided quaternionic vector space is given by $\bH$ itself.

\begin{remark}
A quaternionic vector space becomes a real vector space when its scalar multiplication is restricted to $\R$.
\end{remark}

One can start with any real vector space $\sfV_\R$ and construct a two-sided quaternionic vector space by setting
\[
\sfV_\bH := \left\{\sum_{i=0}^3 v_i\otimes e_i : v_i\in \sfV_\R\right\},
\]
where $\otimes$ denotes the algebraic tensor product. (See, for instance, \cite{Ng}.) 

On the other hand, given any two-sided quaternionic vector space and defining
\[
\sfV_\R := \{v\in \sfV : \lambda v = v \lambda,\;\forall\lambda\in\bH\},
\]
then $\sfV_\R$ is a real vector space called the \emph{real part of $\sfV$}.

The proof of the next result can be found in, i.e., \cite{Ng}.

\begin{proposition}
Let $\sfV$ be a two-sided quaternionic vector space and let $\sfV_\R$ denote its real part. Then $V\cong \sfV_\R\otimes\bH$.
\end{proposition}

Next, we introduce quaternionic normed spaces.

\begin{definition}
Let $\sfV$ be a left quaternionic vector space. A function $\n{\cdot}: \sfV\to \R_0^+$ is called a norm on $V$ if
\begin{enumerate}
\item $\n{v} = 0$ iff $v=0$.
\item $\n{q v} = \abs{q} \n{v}$, for all $v\in \sfV$ and $q\in \bH$.
\item $\n{v+w} \leq \n{v} + \n{w}$, for all $v,w\in \sfV$.
\end{enumerate}
A left quaternionic vector space endowed with a norm will be called a left quaternionic normed space. 
\end{definition}

A left quaternionic normed space $\sfE$ is called \emph{complete} if it is a complete metric space with respect to the metric $d(x,y) = \n{x-y}$ induced by the norm $\n{\cdot}_\sfE$. In this case, we refer to $\sfE$ as \emph{left quaternionic Banach space}.

\begin{remark}
A left (or right) quaternionic Banach space becomes a real Banach space if the left (right) scalar multiplication is restricted to $\R$. (Cf., \cite[Section 2.3]{CGK}.)
\end{remark}

\begin{example}
The space $\bH^k$ consisting of $k$-tuples of quaternions is both a left and a right quaternionic vector space. We represent elements $\xi\in\bH^k$ as column vectors and define the quaternionic conjugate ${}^*$ of $\xi$ by
\[
\begin{pmatrix} \xi_1\\ \vdots\\ \xi_k
\end{pmatrix}^* := \begin{pmatrix} \overline{\xi_1} & \cdots & \overline{\xi_k}
\end{pmatrix},
\]
where each $\xi_j\in\bH$. When endowed with the norm
\be\label{dist}
\n{\xi}_k := \sqrt{\xi^*\xi} = \sqrt{\sum\limits_{l=1}^k \abs{\xi_j}^2},
\ee
$\bH^k$ becomes a two-sided quaternionic Banach space as $\lambda v = v \lambda$, for all $\lambda\in\R$ and for all $v\in\bH^k$.

Note that under the norm $\n{\cdot}_k$, $\bH^k$ becomes a topological space and also a complete metric space.
\end{example}

The final concept we need to introduce is that of left linear mapping between left quaternionic vector spaces.

\begin{definition}
Let $\sfV_1$ and $\sfV_2$ be left quaternionic vector spaces. A mapping $f:\sfV_1\to\sfV_2$ is called left linear if
\[
f(q v + w) = q f(v) + f(w), \quad\forall v,w\in \sfV, \forall q\in \bH.
\]
A left linear mapping is called bounded if
\[
\n{f} := \sup_{x,y \in \sfV_1, x \neq y} \frac{\n{f(x)-f(y)}_{\sfV_2}}{\n{x-y}_{\sfV_1}} < \infty.
\]
\end{definition}
\subsection{Quaternionic Fractal Interpolation}
In this section, we introduce the novel concept of quaternionic fractal interpolation. For illustrative purposes, we do not choose the most general set up but restrict ourselves to the case where the left quaternionic Banach spaces are $\sfE := \bH^k =:\sfF$, $k\in \N$.

For our purposes, we need the following function space. Let $\sfX\subset\bH^k$ be compact (as defined via the norm $\n{\cdot}_k$) and let
\[
\cB(\sfX,\bH^k) := \left\{f:\sfX\to\bH^k : \text{$f$ is bounded}\right\}.
\]
A function $f:\sfX\to\bH^k$ is called \emph{bounded} if there exists a real number $M>0$ such that $\n{f}_k \leq M$. If we define for $x\in\sfX$ and $\lambda\in\bH$
\[
(f+g) (x) := f(x) + g(x)\quad\text{and}\quad (\lambda\cdot f)(x) := \lambda\cdot f(x)
\]
then $\cB(\sfX,\bH^k)$ becomes a left quaternionic vector space. Setting for each $f\in \cB(\sfX,\bH^k)$
\[
\n{f} := \sup_{x\in\sfX} \n{f(x)}_k,
\]
then $\cB(\sfX,\bH^k)$ becomes a left Banach space. (See, for instance, \cite[Chapter IV.E.1.]{J} for the case $k=1$. The extension to $k > 1$ is straight-forward.)

For the sake of simplicity and the purpose of understanding the underlying issues, we concentrate on the special case that 
\[
\sfX := \left\{q\in \bH : \max\limits_{i=0,1,2,3} \abs{q_i} = 1\right\} \cong [-1,1]^4
\]
and that $k:=1$. The interested reader is encouraged to consider the general set-up.

We take the four-dimensional cube $\sfX$ and divide it into $n:=2^4$ congruent four-dimensional subcubes $\sfX_i$ each similar to $\sfX$ and such that $\{\sfX_i\}_{i=1}^{n}$ forms a partition of $\sfX$ in the sense of \eqref{c1} and \eqref{c2}. We leave it to the diligent reader to derive closed expressions for the $n$ injections $l_i$.

On the left Banach space $\cB(\sfX,\bH)$, we consider the RB operator $T: \cB(\sfX,\bH)\to\cB(\sfX,\bH)$ given by
\be\label{6.5}
Tf(l_i(x)) := q_i(x) + s_i(x) f(x), \quad x\in \sfX, \;\;i \in \N_n,
\ee
where $q_i, s_i:\sfX\to\bH$ are bounded functions. Clearly, $T f\in \cB(\sfX,\bH)$ under these assumptions. Note that we left-multiply $f$ by $s_i$. 

\begin{remark}
In the case of real Banach spaces, the RB operator is affine, i.e., $T - T(0)$ is a linear operator. In the quaternionic setting this is no longer true: $T - T(0)$ is not a left linear operator. It is only if $s_i (x) \in \R$.
\end{remark}

To simplify notation, we introduce the following abbreviation for the $m$-fold composition of functions from an IFS. Let $\cF:=\{f_1, \ldots, f_n\}$. We write
\[
f_{i_m i_{m-1} \cdots i_1} := f_{i_m}\circ f_{i_{m-1}}\circ f_{i_1},
\]
where each $i_j\in \N_n$.

For each $m\in \N$, the $m$-fold application of $T$, $T^m f := T(T^{m-1} f)$, can be written as
\[
T^m f (l_{i_m i_{m-1} \cdots i_1}(x)) = \sum_{k=1}^{m} \prod_{j=1}^{k-1} s_{i_j} (x) q_{k} (x)+ \prod_{k=1}^m s_{i_k}(x) f(x),
\]
where the factors in the products $\prod$ are left-multiplied and where we set the empty product equal to 1. Notice the reverse order of the indices $i_m i_{m-1} \cdots i_1$ on the left- and right-hand side.

Continuing as in the previous section to ensure the existence of a fixed point $\psi$ for the RB operator \eqref{6.5} or a bounded solution of the associated system of functional equations, we obtain in summary the next result.
\begin{theorem}
For the above setting, the RB operator $T$ defined in \eqref{6.5} has a unique fixed point $\psi\in\cB(\sfX,\bH)$, i.e.,
\be
T\psi = \psi\quad\Longleftrightarrow\quad\psi (l_i(x)) = q_i(x) + s_i(x) \psi(x), \quad x\in \sfX, \;\; i \in \N_n,
\ee
provided that 
\[
\max\limits_{i\in \N_n} \sup\limits_{x\in \sfX_i} |s_i(x)| < 1.
\]
\end{theorem}
The fixed point $\psi$ is then called a \emph{bounded quaternionic fractal function}. 

\begin{remark}
A similar result exists of course for right Banach spaces and RB operators where the functions $s_i$ are right-multiplied onto $f$. Note that the fixed point $\psi$ depends on this right or left multiplication.
\end{remark}

\begin{example}
This example shows the versatility of quaternionic fractal interpolation. We choose as $\sfX := \{q\in\bH : \Sc q \in [0,1) \wedge \Ve q = 0\}$. Note that $\sfX \cong [0,1] \subset\R$. Define injections $l_i:\sfX\to\sfX$ as follows:
\[
l_1 (x) : = \tfrac12 x\quad \text{and} \quad l_2 (x):= \tfrac12(x+1).
\]
Clearly, $\{l_1(\sfX), l_2(\sfX)\}$ is a partition of $\sfX$ satisfying \eqref{c1} and \eqref{c2}.

Moreover, let $q_1 := e_0 + 2e_1-e_3+3e_4$ and $q_2 := -e_0 - 2e_1 + 2e_3+e_4$ be two quaternions. We set
\[
q_1 (x) := (1-q_1) x \quad \text{and} \quad q_2(x) := q_2 x^2. 
\]

Then, $q_1,q_2\in \cB(\sfX,\bH)$. We define an RB operator $T$ by
\[
Tf (x) := \begin{cases} 2(1-q_1) x + s_1 f(2x), & x\in [0,\frac12),\\
q_2(2x-1)^2 + s_2 f(2x-1), & x\in [\frac12,1),
\end{cases}
\]
where $s_1 := \frac{1}{10}e_0 + \frac12 e_1 - \frac15 e_2 - \frac{1}{10} e_3$ and $s_1 := -\frac{1}{5}e_0 + \frac15 e_1 - \frac35 e_2 + \frac{1}{10} e_3$. Note that $\abs{s_1} = \frac{1}{10}\sqrt{31}$ and $\abs{s_2} = \frac{1}{10}\sqrt{45}$. Thus, $\max\{s_1,s_2\} < 1$.

Therefore, the RB operator $T$ is a contraction on $\cB(\sfX,\bH)$ and has a unique fixed point $\psi\in\cB(\sfX,\bH)$. The following Figure \ref{fig5} shows the projection of the graph of $\psi$ onto the $(e_0,e_i)$-plane for $i=0,1,2,3$.

\begin{figure}[h!]
\begin{center}
\includegraphics[width=5cm, height=4cm]{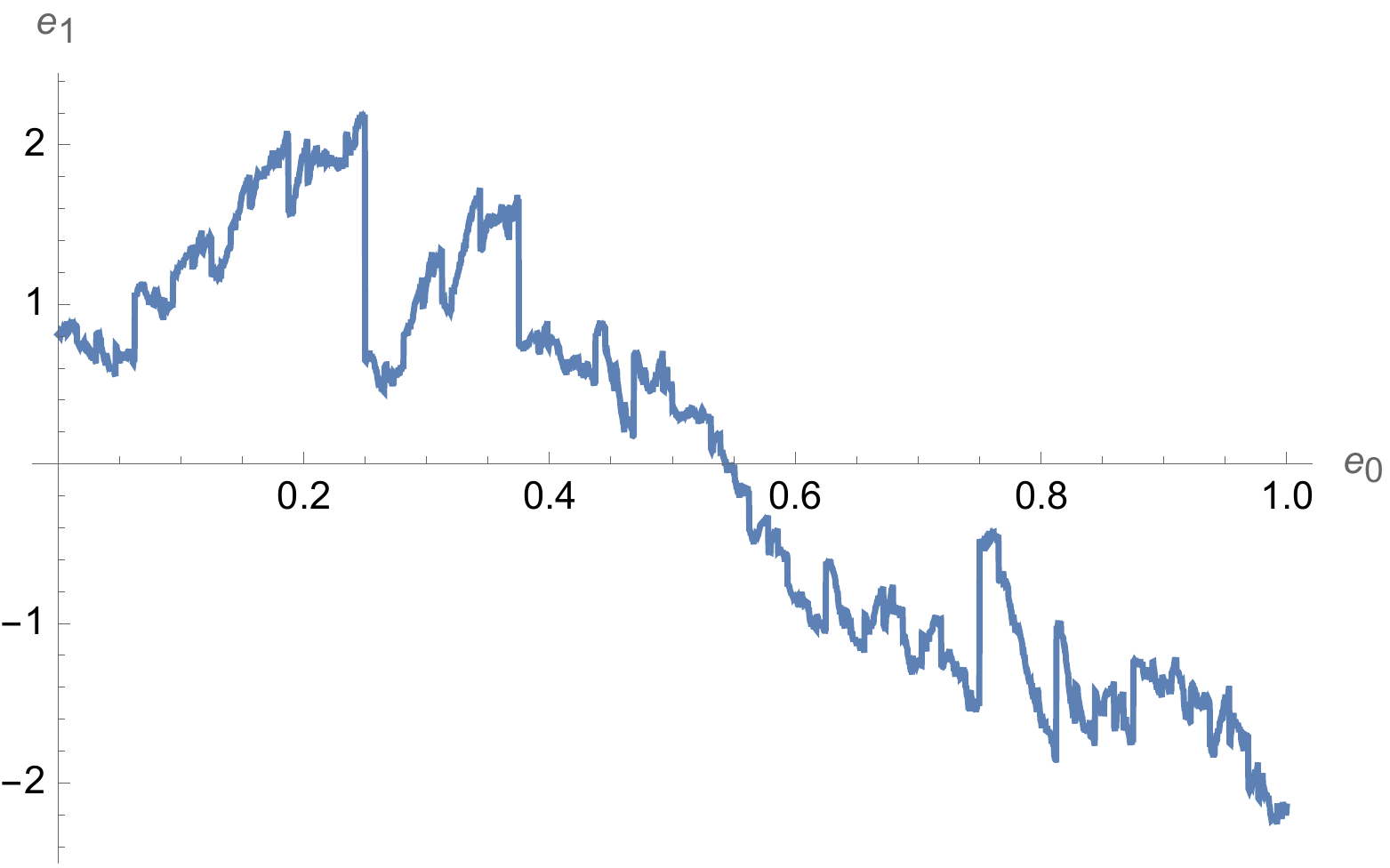}\hspace*{1cm}\includegraphics[width=5cm, height=4cm]{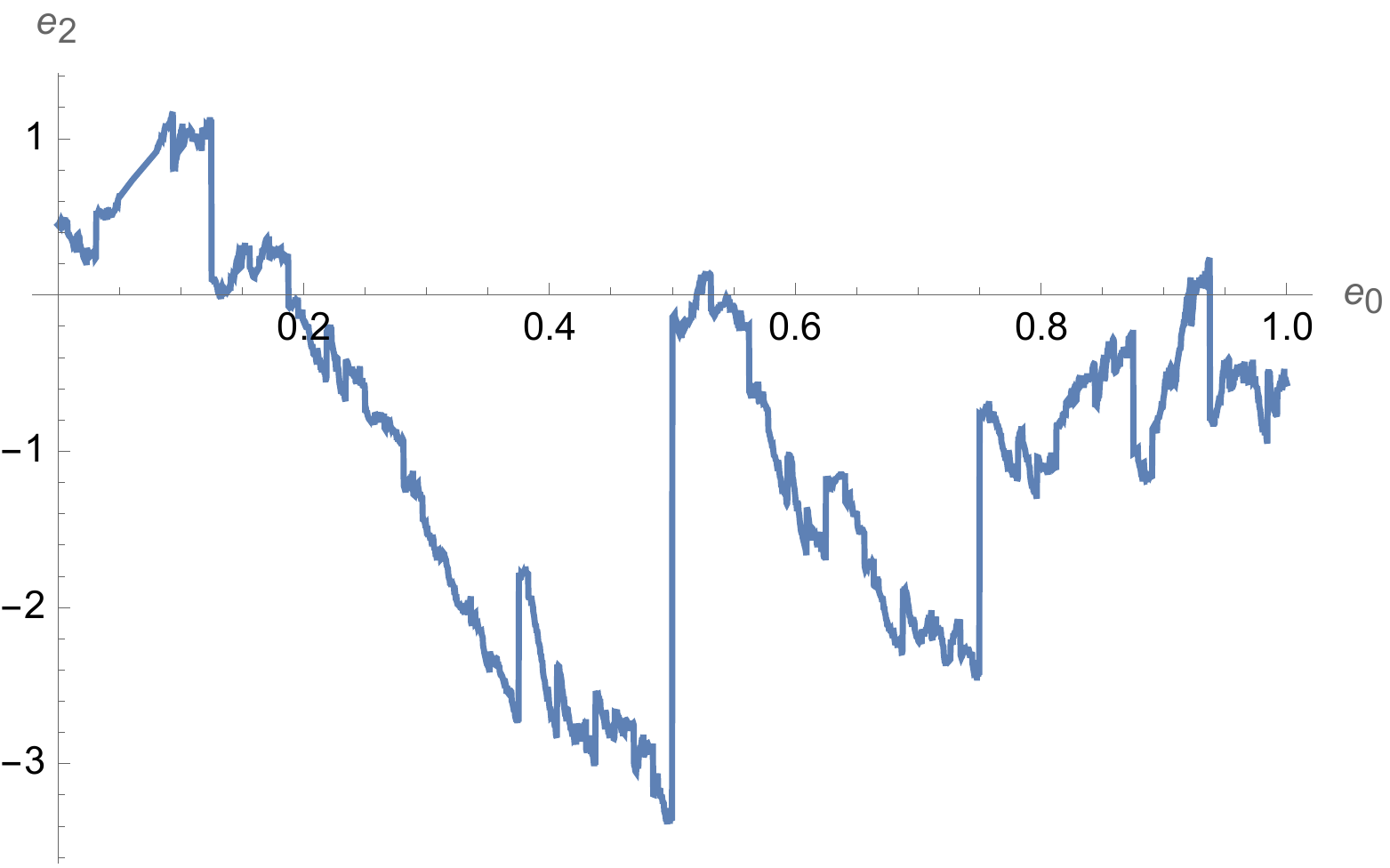}\\
\includegraphics[width=5cm, height=4cm]{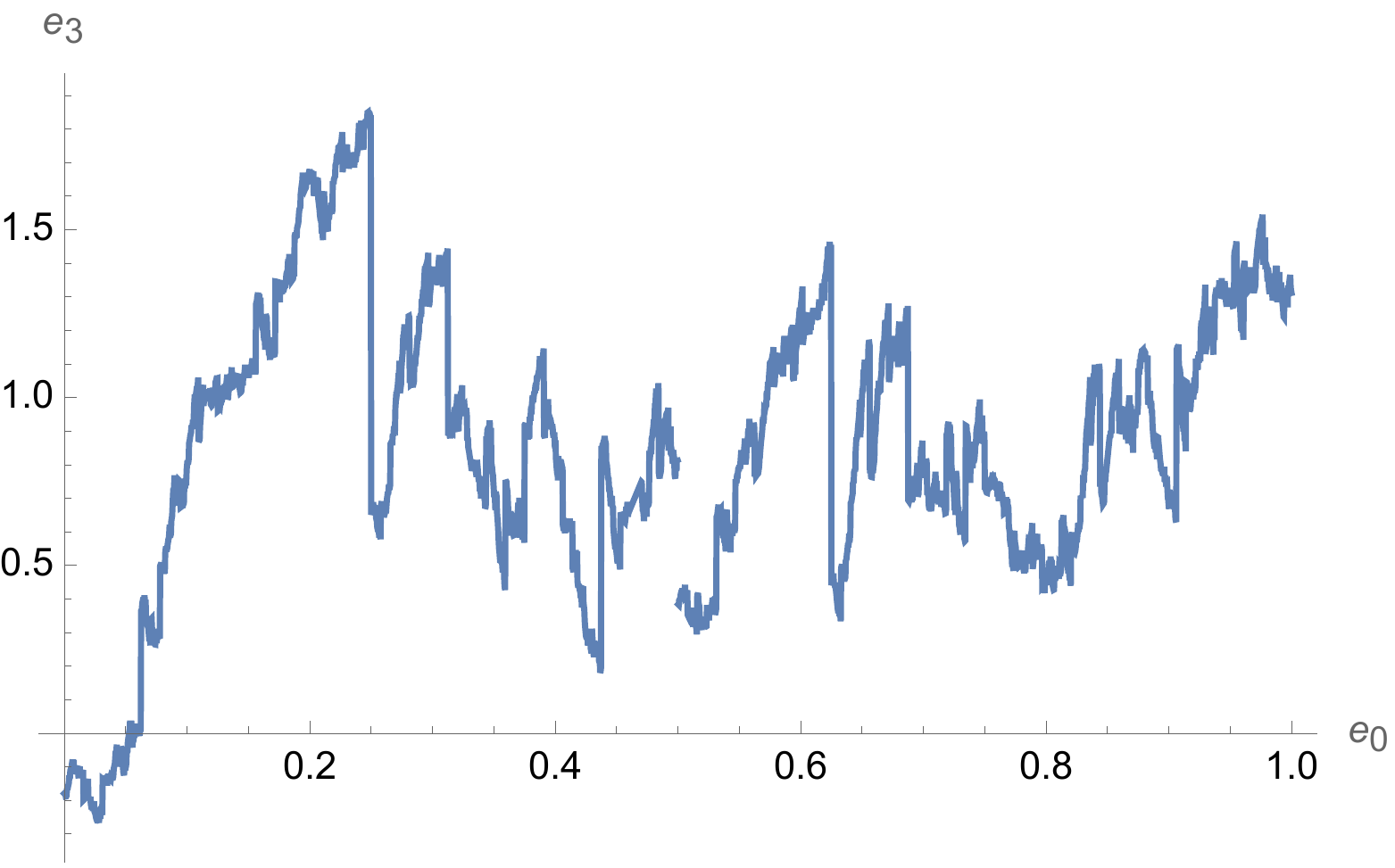}\hspace*{1cm}\includegraphics[width=5cm, height=4cm]{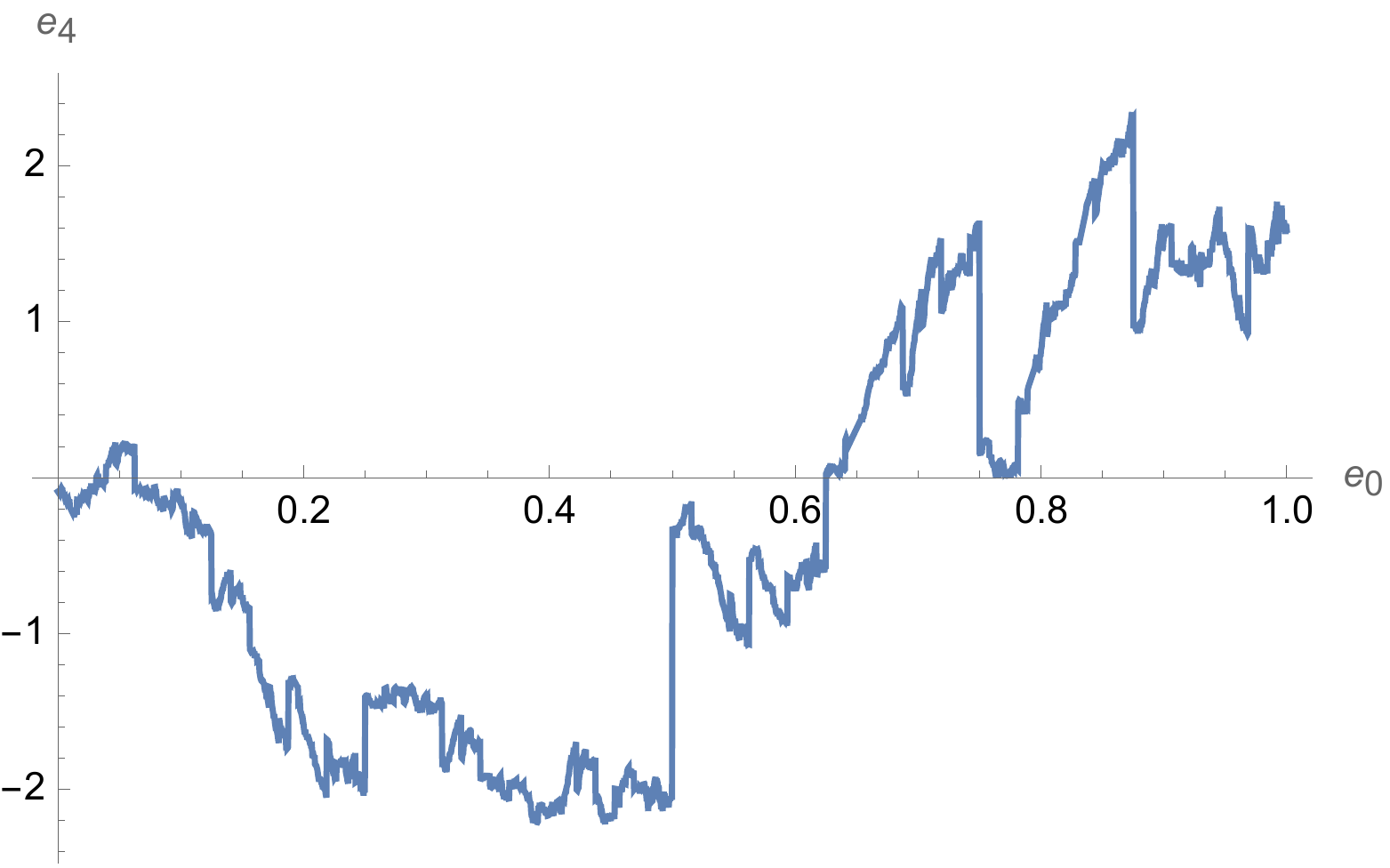}
\end{center}
\caption{The projections of $\psi$ onto the $(e_0,e_i)$-planes.}\label{fig5}
\end{figure}
As $\psi$ can be written as $\psi = \sum\limits_{i=0}^3 \psi_i e_i$, we display in Figure \ref{fig6} the parametric plots $(\psi_0, \psi_1, \psi_2)$ and $(\psi_0, \psi_2, \psi_4)$.
\begin{figure}[h!]
\begin{center}
\includegraphics[width=5cm, height=4cm]{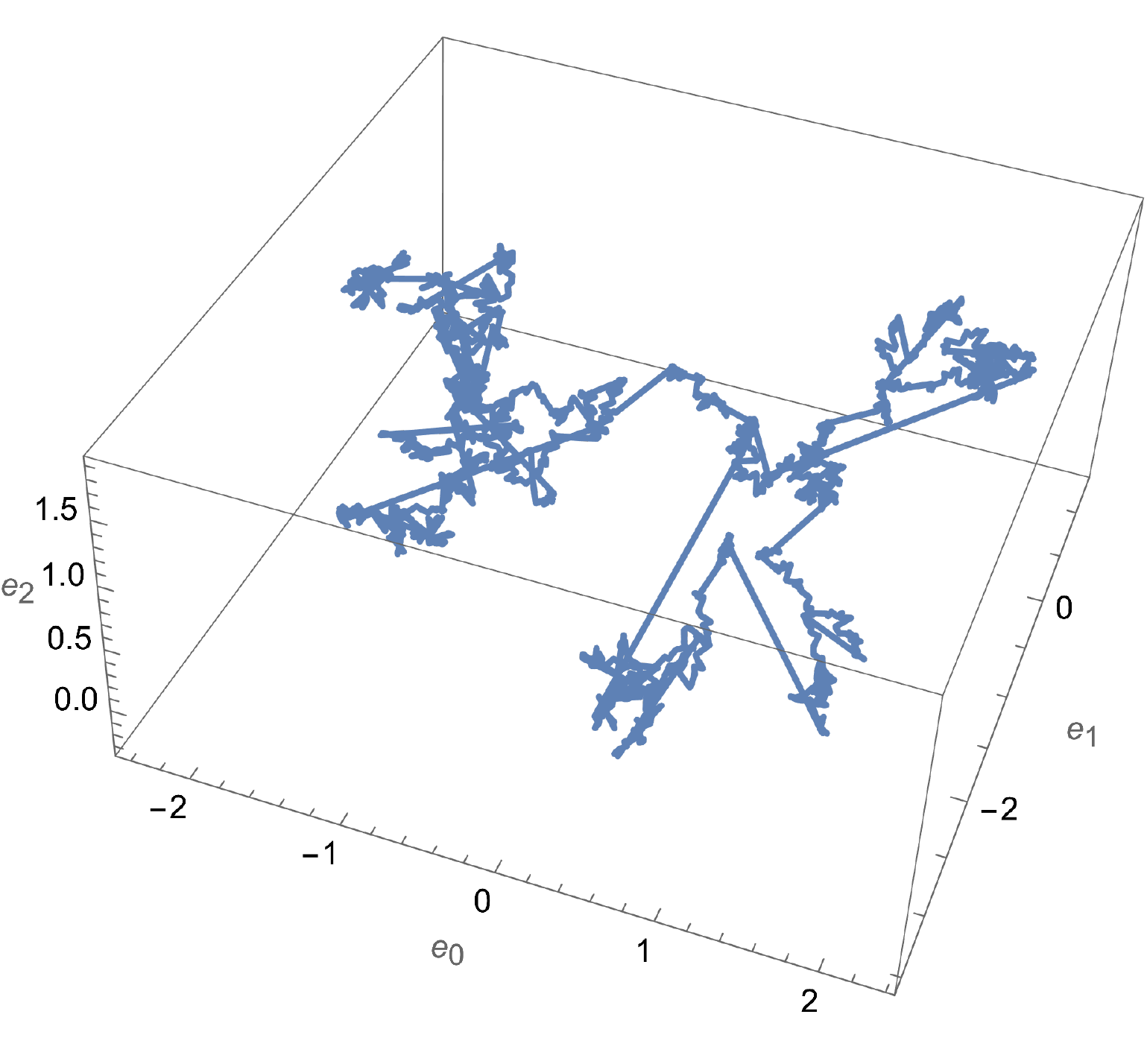}\hspace*{1cm}\includegraphics[width=5cm, height=4cm]{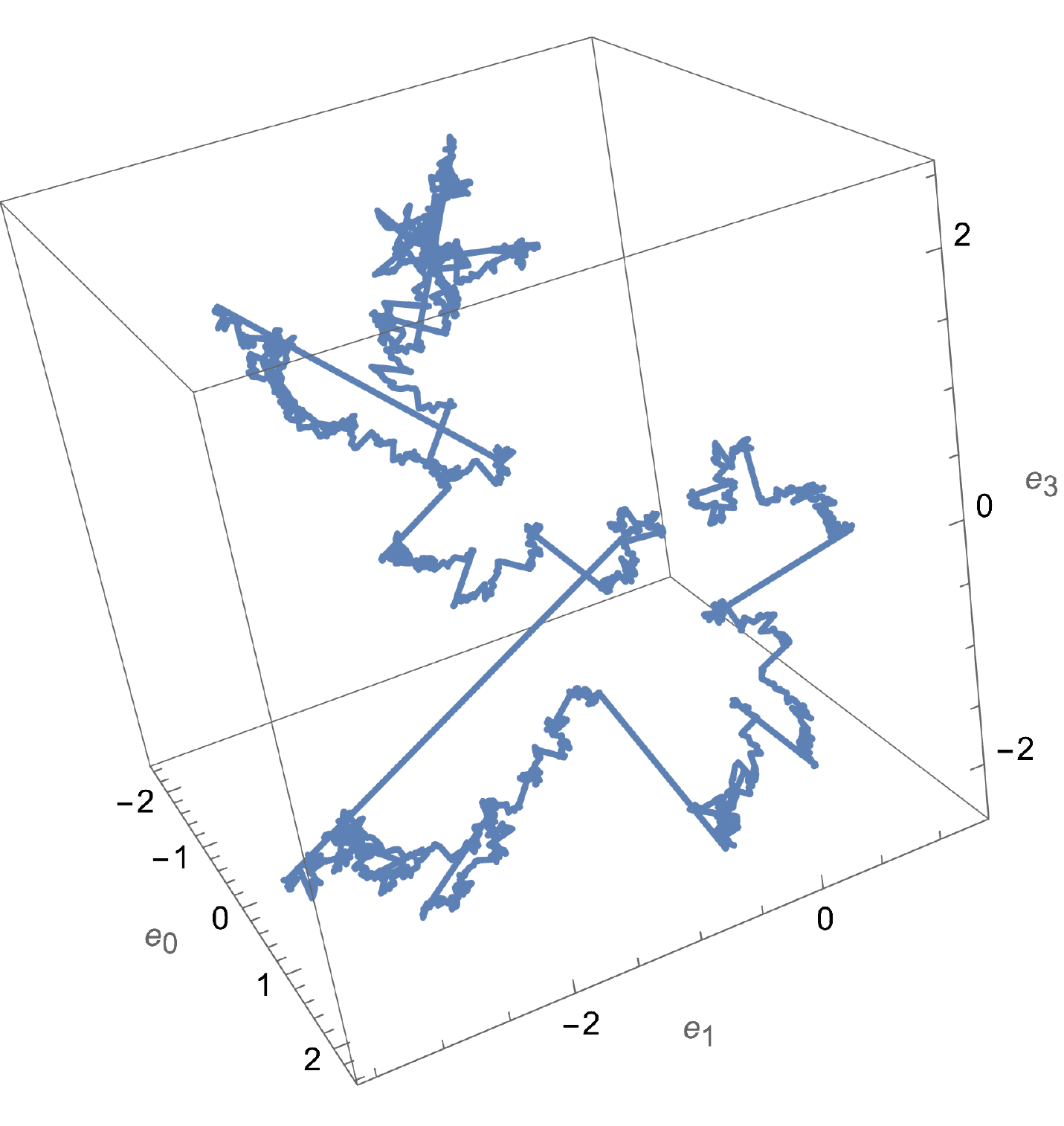}\end{center}
\caption{Some parametric plots of the components of $\psi$.}\label{fig6}
\end{figure}
\end{example}
\section{Summary}
In this chapter, we considered several aspects of fractal interpolation and identified their commonalities. Fractal interpolation is based either on solving a system of functional equations \cite{SB,SB1} or on obtaining the fixed point of an operator \cite{massopust,massopust1}. Both approaches yield the same results and are interchangeable. 

In particular, we presented
\begin{itemize}
\item global fractal interpolation. This was the original approach undertaken in \cite{B2} which is based on a geometric setting. The analytic setting based on an RB operator commenced in \cite{bedford,dubuc1,dubuc2} and produced numerous generalizations and extensions. The main idea is to use a fixed bounded subset $\sfX$ of a normed space $\sfE$ to produce a partition of $\sfE$ over which copies of functions defined on the global set $\sfX$ reside. Each of the partitioning sets is then iteratively partitioned into subsets thus producing the fractal nature of the limiting object.
\item local fractal interpolation. Here, we allow subsets of a bounded subset $\sfX$ of a normed space $\sfE$ to produce a partition of $\sfX$. These subsets can be repeated in the construction of the partition and need no longer be defined by contractive injections. Here, functions defined on the local subsets reside over partitions of $\sfX$. This type of local interpolation was successfully employed in fractal compression and image analysis. (Cf., for instance, \cite{BH}.)
\item non-stationary fractal interpolation. In this approach, we do no longer require that we keep the same quantities at each level of iteration but allow them to vary. This set-up is very similar to non-stationary subdivision and was investigated in this respect in \cite{DLM,LDV,m4}. This new methodology shows great potential for future investigations into the subject. 
\item quaternionic fractal interpolation. In this novel approach, we kept the set-theoretic structure but allowed for a change in the algebraic structure of the underlying sets. By requiring that we are using a division algebra such as the quaternions or more generally hypercomplex numbers, we obtain even more flexibility in the construction of fractal interpolants. The investigation into this type of setting commenced in \cite{m5} and promises to be a fruitful field of investigation in years to come.
\end{itemize}


\begin{thebibliography}{99}
%
\bibitem {B2} Barnsley, M.F. Fractal functions and interpolation. \textit{Constr. Approx.} \textbf{1986}, \textit{2}, 303--329.
%
\bibitem{BHVV} Barnsley, M.F., Harding, B., Vince, C., Viswanathan, P. Approximation of rough functions. \textit{J. Approx. Th.} \textbf{2016}, \textit{209}, 23--43. 
%
\bibitem{bhm} Barnsley, M.F., Hegland, M., Massopust, P.R. Numerics and Fractals. \textit{Bull. Inst. Math. Acad. Sin. (N.S.)} \textbf{2014}, \textit{9(3)}, 389--430.
%
\bibitem{bhm2} Barnsley, M.F., Hegland, M., Massopust, P.R. Self-referential functions. https://arxiv.org/abs/1610.01369

\bibitem{BH} Barnsley, M.F., Hurd, L.P. \textit{Fractal Image Compression}. AK Peters Ltd., Wellesly, MA, 1993.
%
\bibitem{bedford} Bedford, T., Dekking, M., Keane, M. Fractal image coding techniques and contraction operators. \textit{Delft University of Technology Report} \textbf{1992}, \textit{92-93}.
%
\bibitem{Bour} Bourbaki, N. \textit{\'El\'ements de math\'ematiques. Alg\`{e}bre. Chapitres 1 \'a 3.} Hermann, Paris, 1970.
%
\bibitem{CGK} Colombo, F.; Gantner, J.; Kimsey, D.P. \textit{Spectral Theory on the $S$-Spectrum for Quaternionic Operators}. Operator Theory: Applications and Advances, Vol. 270. Birkh\"auser Verlag, Switzerland, 2010.
%
\bibitem{DLM} Dira, N., Levin, D., Massopust, P. Attractors of trees of maps and of sequences of maps between spaces and applications to subdivision. \textit{J. Fixed Point Theory Appl.} \textbf{2020}, \textit{22}(14), 1--24.
%
\bibitem{DS} Dubins, L.E., Savage, L.J. \textit{Inequalities for Stochastic Processes}, Dover Publications: New York, 1976.

\bibitem{dubuc1} Dubuc, S. Interpolation through an iterative scheme. \textit{J. Math. Anal. Appl.} \textbf{1986}, \textit{114(1)}, 185--204.

\bibitem{dubuc2} Dubuc, S. Interpolation fractale. In \textit{Fractal Geomety and Analysis}, J. B\'elais and S. Dubuc, eds., Kluwer Academic Publishers, Dordrecht, The Netherlands, 1989.

\bibitem{GHS} G\"urlebeck, K.; Habetha, K.; Spr\"o{\ss}ig, W. \textit{Holomorphic Functions in the Plane and $n$-dimensional Space}, Birkh\"auser Verlag: Basel, Switzerland, 2000.

\bibitem{GS} G\"urlebeck, K.;  Spr\"o{\ss}ig, W. \textit{Quaternionic and Clifford Calculus for Physicists and Engineers}, John Wiley \& Sons: Chichester, England, 1997.

\bibitem{H} Hutchinson, J.E. Fractals and self-similarity. \textit{Indiana Univ. Math. J.} \textbf{1981}, \textit{30}, 713--747.

\bibitem{J} Jamison, J.E. \textit{Extension of Some Theorems of Complex Functional Analysis to Linear Spaces over the Quaternions and Cayley Numbers} (1970). Doctoral Dissertations. 2037. https://scholarsmine.mst.edu/doctoral\_dissertations/2037.

\bibitem{Kiess} Kiesswetter, K. Ein einfaches Beispiel f\"{u}r eine Funktion welche \"{u}berall stetig und nicht differenzierbar ist. \textit{Math. Phys. Semesterber.} \textbf{1966}, \textit{13}, 216--221.

\bibitem{Krav}  Kravchenko, V. \textit{Applied Quaternionic Analysis}, Heldermann Verlag: Lemgo, Germany, 2003.

\bibitem{LDV} Levin, D.; Dyn, N.; Viswanathan, P. Non-stationary versions of fixed-point theory, with applications to fractals and subdivision. \textit{J. Fixed Point Theory Appl.} \textbf{2019}, \textit{21}, 1--25.

\bibitem{PRM} Massopust, P.R. Fractal functions and their applications. \textit{Chaos, Solitons and Fractals}, \textbf{1997}, \textit{8(2)}, 171--190.
 
\bibitem {massopust} Massopust, P.R. \textit{Interpolation and Approximation with Splines and Fractals,} Oxford University Press: Oxford, USA, 2010.

\bibitem{massopust1} Massopust, P.R. \textit{Fractal Functions, Fractal Surfaces, and Wavelets}, 2nd ed., Academic Press: San Diego, USA, 2016.

\bibitem{m2} Massopust, P.R. Local fractal functions and function spaces. \textit{Springer Proceedings in Mathematics \& Statistics: Fractals, Wavelets and their Applications} \textbf{2014}, Vol. 92, 245--270.

\bibitem{m3} Massopust, P.R. Local Fractal Functions in Besov and Triebel-Lizorkin Spaces. \textit{J. Math. Anal. Appl.} \textbf{2016}, \textit{436}, 393 -- 407. 

\bibitem{m7} Massopust, P.R. Local fractal interpolation on unbounded domains. \textit{Proc. Edinburgh Math. Soc.} \textbf{2016}, \textit{61}, 151--167.

\bibitem{m4} Massopust, P.R. Non-Stationary Fractal Interpolation. \textit{Mathematics} \textbf{2019}, \textit{7(8)}, 1 -- 14.

\bibitem{m5} Massopust, P.R. Hypercomplex Iterated Function Systems. To appear in \textit{Current Trends in Analysis, its Applications and Computation.}
Proceedings of the 12th ISAAC Congress, Aveiro, Portugal, 2019, Cereijeras, P.; Reissig, M.; Sabadini, I.; Toft, J. (eds.), Birkh\"auser.

\bibitem{MGS} Morais, J.P.; Georgiev, S.; Spr\"o{\ss}ig, W. \textit{Real Quaternionic Calculus}, Birkh\"auser Verlag: Basel, Switzerland, 2014.

\bibitem{N} Navascu\'es, M.A. Fractal polynomial interpolation. \textit{Z. Anal. Anwendungen}, \textbf{2005}, \textit{24(2)}, 401--418.

\bibitem{Ng} Ng, C. One quaternionic functional analysis. \textit{Math. Proc. Camb. Phil. Soc.}, \textbf{2007}, \textit{143}, 391--406.

%\bibitem{R} Rolewicz, S. \textit{Metric Linear Spaces}, Kluwer Academic: Warsaw, Poland, 1985.
%
%\bibitem{rudin} Rudin, W. {\em Functional Analysis}, McGraw--Hill: New York, 1991.
%
\bibitem{SB} Serpa, C.; Buesca, J. Constructive solutions for systems of iterative functional equations. \textit{Constructive Approx.}, \textbf{2017}, \textit{45(2)}, 273--299.

\bibitem{SB1} Serpa, C.; Buesca, J. Compatibility conditions for systems of iterative functional equations with non-trivial contact sets. \textit{Results Math.} \textbf{2021}, \textit{2}, 1--19.

\bibitem{Tak} Takagi, T. A simple example of the continuous function without derivative. {\it Proc. Phys. Math. Soc. Japan} \textbf{1903}, {\it 1}, 176--177.
\end{thebibliography}
\end{document}